\documentclass{article}
\usepackage[utf8]{inputenc}
\usepackage{a4wide}
\usepackage[UKenglish]{babel}
\usepackage{graphicx}
\usepackage{subcaption}
\usepackage[colorlinks=true,linkcolor=blue,citecolor=red]{hyperref}
\usepackage{amssymb,empheq,bbold}
\usepackage[inline]{enumitem}
\usepackage{geometry}
\usepackage{caption}
\usepackage{amsthm}
\usepackage{bbm}
\usepackage[title]{appendix}
\usepackage[mathlines]{lineno}
\usepackage{overpic}

\graphicspath{{figures/}}

\newcommand \commentout[1] {}

\newcommand{\R}{\mathbb{R}}

\newtheorem{theorem}{Theorem}
\newtheorem{definition}[theorem]{Definition}
\newtheorem{proposition}[theorem]{Proposition}
\newtheorem{lemma}[theorem]{Lemma}
\theoremstyle{definition}
\newtheorem{remark}{Remark}

\title{Continuum limit of hypergraph $p$-Laplacian equations on point clouds}

\author{
Kehan Shi\thanks{{Department of Mathematics, China Jiliang University, Hangzhou 310018, China.} \texttt{kshi@cjlu.edu.cn}}
}
\date{}

\begin{document}

\maketitle

\begin{abstract}
This paper studies a class of $p$-Laplacian equations on point clouds that arise from hypergraph learning in a semi-supervised setting.
Under the assumption that the point clouds consist of independent random samples drawn from a bounded domain $\Omega\subset\mathbb{R}^d$,
we investigate the asymptotic behavior of the solutions as the number of data points tends to infinity, with the number of labeled points remains fixed.
We show, for any $p>d$ in the viscosity solution framework, that the continuum limit is a weighted $p$-Laplacian equation subject to mixed Dirichlet and Neumann boundary conditions.
The result provides a new discretization of the $p$-Laplacian on point clouds.
Numerical experiments on high-dimensional data interpolation are presented.
\end{abstract}

\medskip


\medskip


\section{Introduction}
Semi-supervised learning is a fundamental problem in data science and has tremendous practical value. In many tasks, the acquisition of labeled data is expensive and time-consuming, whereas the acquisition of unlabeled data is relatively easy.
This requires developing reliable algorithms to assign labels to a large number of unlabeled data points using a small number of given labeled ones \cite{zhu2009introduction}.
It is mathematically formulated as the following interpolation problem.
We are given a set of points $\Omega_n=\{x_i\}_{i=1}^n$ drawn from a bounded domain $\Omega\subset\mathbb{R}^d$ and a function $f$ defined on a subset $\mathcal{O}=\{x_i\}_{i=1}^N$, where $1\leq N<n$ and the sampling rate $\frac{N}{n}$ is usually low.
The goal is to find an estimator $u_n\in\mathbb{R}^n$ satisfying $u_n(x_i)=f(x_i):=y_i$ for $x_i\in\mathcal{O}$.

Additional assumptions are needed to ensure a meaningful solution for the underdetermined inverse problem.
A common approach is to adopt the smoothness assumption in the graph setting.
The data set together with its geometric structure are represented by a graph $G_n=(V_n, W_n)$ with vertex set $V_n=\Omega_n$ and edge weight set $W_n$.
A nonnegative weight $ w_{i,j}\in W_n$ denotes the similarity between points $x_i$ and $x_j$, which is usually a decreasing function of the distance $|x_i-x_j|$.
An edge $e_{i,j}$ connecting $x_i$ and $x_j$ exists if and only if $w_{i,j}>0$.

The smoothness assumption asserts that connected points in $G_n$ share similar values. This can be achieved by regularization on the graph.
Considering the $p$-Laplacian regularization,
 we find a solution $u_n$ by solving the constrained optimization problem
\begin{equation*}
  \min_{u_n\in\mathbb{R}^n}F^{G_n}_p(u_n):=\sum_{i,j=1}^nw_{i,j}|u_n(x_i)-u_n(x_j)|^p, \quad \mbox{ s.t. $u_n=f$ on $\mathcal{O}$},
\end{equation*}
where $1<p<\infty$.
By gradient descent, $u_n$ satisfies the $p$-Laplacian equation
\begin{equation}\label{eq:1.5}
  \begin{cases}
    L^{G_n}_pu_n(x_i):=\sum\limits_{j=1}^{n}w_{i,j}|u_n(x_j)-u_n(x_i)|^{p-2}(u_n(x_j)-u_n(x_i))=0, & \mbox{if } x_i\in\Omega_n\backslash\mathcal{O}, \\
    u_n(x_i)=y_i, & \mbox{if } x_i\in \mathcal{O}.
  \end{cases}
\end{equation}
If $p=2$, $F^{G_n}_p$ is the graph version of the well-known Tikhonov regularization and $L^{G_n}_p$ is the graph Laplacian.
They were first proposed for semi-supervised learning in \cite{zhu2003semi} and subsequently used in a variety of tasks, including image processing and machine learning \cite{von2007tutorial,gilboa2007nonlocal,elmoataz2015p,shi2017weighted}.
Although the graph Laplacian equation is easy to solve (since it corresponds to a linear system), it is often preferred to use the $p$-Laplacian with a large exponent $p>2$ in applications when the sampling rate $\frac{N}{n}$ is extremely low \cite{el2016asymptotic,flores2022analysis}.
This is due to the fact that, for a random geometric graph $G_n$, when $N$ is fixed and $n\rightarrow\infty$, the solution of the graph Laplacian equation  converges to a constant on $\Omega_n\backslash\mathcal{O}$ \cite{calder2020properly}.
Since $u=f$ on $\mathcal{O}$, the spiky phenomenon is observed at $\mathcal{O}$.
See Figure \ref{fig:1-a} for an illustration.
It has been shown that $p>d$ is a necessary condition for avoiding the spiky solution \cite{slepcev2019analysis}. Formally speaking, $F^{G_n}_p(u_n)$ is consistent with an energy of the form $\int_\Omega|\nabla u|^pdx$, which is well-defined in $W^{1,p}(\Omega)$.
If $p>d$, by Sobolev's embedding theorem, $u\in W^{1,p}(\Omega)\hookrightarrow C^{\alpha}(\overline{\Omega})$ is H\"{o}lder continuous and thus the  spiky solution is prevented in the discrete case.
Figure \ref{fig:1-a}--\ref{fig:1-c} show an example of the spiky phenomena, which diminishes as the exponent $p$ increases.

A main drawback of the graph $p$-Laplacian with $p> 2$ is its low numerical efficiency, especially when $p$ is large.
It is straightforward to let $p\rightarrow+\infty$ and
consider the Lipschitz learning \cite{von2004distance,kyng2015algorithms}
\begin{equation*}
  F^{G_n}_{\infty}(u_n)=\max_{x_i,x_j\in\Omega_n}w_{i,j}|u_n(x_i)-u_n(x_j)|.
\end{equation*}
The asymptotic behavior when $N$ is fixed and $n\rightarrow\infty$ has been studied in \cite{roith2023continuum}, where its well-posedness for semi-supervised learning is established.
Note that $F^{G_n}_{\infty}$ only penalizes the largest gradient. The minimizer is not unique in general.
In practice, one looks for a minimizer whose gradient is minimal in the lexicographical order \cite{kyng2015algorithms}.
More precisely, the minimizer satisfies the graph $\infty$-Laplace equation
\begin{equation*}
  \begin{cases}
    L^{G_n}_{\infty}u_n(x_i):=\max\limits_{1\leq j\leq n}w_{i,j}(u_n(x_j)-u_n(x_i))+\min\limits_{1\leq j\leq n}w_{i,j}(u_n(x_j)-u_n(x_i))=0, & \mbox{if } x_i\in\Omega_n\backslash\mathcal{O}, \\
    u_n(x_i)=y_i, & \mbox{if } x_i\in \mathcal{O}.
  \end{cases}
\end{equation*}
It is also well-defined with an arbitrarily low sampling rate \cite{calder2019consistency,bungert2023uniform}.
See Figure \ref{fig:1-c} for a numerical example.
We note that neither the Lipschitz learning nor the graph $\infty$-Laplace equation uses the data distribution of the given data to improve the learning result, which contrasts with the graph $p$-Laplacian with finite $p$ \cite{roith2023continuum,calder2019consistency}.

An alternative is to combine the Laplacian and the $\infty$-Laplacian, leading to the game-theoretic $p$-Laplacian
\begin{equation}\label{eq:1.6}
  \begin{cases}
    \mathcal{L}^{G_n}_pu_n(x_i):=\frac{1}{d_n(x_i)}L^{G_n}_2u_n(x_i)+\lambda (p-2)L^{G_n}_\infty u_n(x_i)=0, & \mbox{if } x_i\in\Omega_n\backslash\mathcal{O}, \\
    u_n(x_i)=y_i, & \mbox{if } x_i\in \mathcal{O},
  \end{cases}
\end{equation}
where $d_n(x_i)=\sum_{j=1}^nw_{i,j}$ and $\lambda>0$ is a constant.
This equation arises in two-player stochastic tug-of-war games \cite{peres2009tug} and has recently been studied in \cite{calder2018game,flores2022analysis} for semi-supervised learning with low sampling rates.
Figure \ref{fig:1-d} provides a numerical example.
Although equation \eqref{eq:1.5} and equation \eqref{eq:1.6} are different at the graph level, they are, after being appropriately rescaled, identical to the continuum $p$-Laplacian equation
\begin{equation}\label{eq:1.7}
    \begin{cases}
    \mbox{div}(\rho^2|\nabla u|^{p-2}\nabla u)=0, & \mbox{in } \Omega\backslash\mathcal{O}, \\
    u=f, & \mbox{on }  \mathcal{O},\\
  \frac{\partial u}{\partial \vec{n}}=0, & \mbox{on } \partial\Omega,
  \end{cases}
\end{equation}
in the limit $n\rightarrow\infty$ with $N$ fixed.
Here $p>d$ and the solution of equation \eqref{eq:1.7} is H\"{o}lder continuous, which ensures that the Dirichlet condition $u=f$ on $\mathcal{O}$ is well-defined.
The weight function $\rho$ is the probability density of the data.

In this paper, we consider the smoothness assumption in the hypergraph setting for semi-supervised learning.
A hypergraph is a generalization of a graph in the sense that each edge (called a hyperedge in a hypergraph) is a subset of the vertex set and can join any number of vertices.
If each hyperedge contains exactly two vertices, it becomes a graph.
Hypergraphs have attracted increasing attention in data science because of its flexibility in modeling higher-order relations in data with applications in recommender systems \cite{tan2011using}, image retrieval \cite{huang2010image}, bioinformatics \cite{patro2013predicting}, etc.

Let $H_n=(V_n, E)$ be an undirected hypergraph with vertex set $V_n=\Omega_n=\{x_i\}_{i=1}^n$ and hyperedge set $E=\{e_k\}_{k=1}^m$, where $e_k\subset V_n$.
Similar to the graph case,
the smoothness assumption asserts that vertices in the same hyperedge tend to share similar values.
This can be achieved by the hypergraph $p$-Laplacian regularization. We then find a solution $u_n$ by solving the constrained optimization problem
\begin{equation}\label{eq:1.2}
  \min_{u_n\in\mathbb{R}^n}
  F_p^{H_n}(u_n):=\sum_{k=1}^m\max_{x_i,x_j\in e_k}|u_n(x_i)-u_n(x_j)|^p, \quad \mbox{ s.t. $u_n=f$ on $\mathcal{O}$},
\end{equation}
where $1<p<\infty$.
The hypergraph $p$-Laplacian regularization $F_p^{H_n}$ with $p=1$ (i.e., the hypergraph total variation) was first proposed in \cite{hein2013total} for data clustering via the Lov\'{a}sz extension of the hypergraph cut.
It is then generalized to the case $p>1$ for semi-supervised learning.
In some applications, the hyperedge set $E$ is not given. It is required to construct the hyperedge set from the data set $V_n$.
The distance-based method \cite{shi2024hypergraph} is a commonly used approach.
More precisely,
\begin{equation}\label{eq:1.1}
  E=\{e_k\}_{k=1}^n,\quad e_k=\{x_j\in\Omega_n: |x_k-x_j|\leq \varepsilon_n\},
\end{equation}
for a constant $\varepsilon_n>0$.
With the hyperedge construction $\eqref{eq:1.1}$, problem \eqref{eq:1.2} is now defined on the point cloud $\Omega_n$.
A direct comparison between the graph learning \eqref{eq:1.5} and the hypergraph learning \eqref{eq:1.2} becomes possible.
It is shown that the hypergraph $2$-Laplacian \eqref{eq:1.2} outperforms the graph $2$-Laplacian \eqref{eq:1.5} and the hypergraph structure is beneficial for handling point cloud data \cite{shi2024hypergraph}. The main limitation of this approach lies in the complexity of the numerical algorithm and its high computational cost, due to the non-differentiability and large scale of the objective function $F_p^{H_n}$.

In \cite{shi2024hypergraph1},
a hypergraph equation of the following form is proposed as an alternative to \eqref{eq:1.2}:
\begin{align}\label{eq:1.3}
  \begin{cases}
    L^{H_n} u_n(x_i):=\sum\limits_{e\in E} \chi_e(x_i)
    \left(\max\limits_{x_j\in e}u_n(x_j)+\min\limits_{x_j\in e}u_n(x_j)-2u_n(x_i)\right)=0,\quad &x_i\in \Omega_{n}\backslash\mathcal{O},\\
    u_n(x_i)=y_i,\quad &x_i\in \mathcal{O},
  \end{cases}
\end{align}
where
\begin{align*}
  \chi_{e}(x_i)=
  \begin{cases}
    1,\quad &\mbox{if } x_i\in e, \\
    0,\quad & \mbox{if } x_i\notin e.
  \end{cases}
\end{align*}
The operator $L^{H_n}$ arises as an approximation and simplification of the multi-valued subgradient $\partial F_2^{H_n}$ of the hypergraph Laplacian $F_2^{H_n}$.
There exist simple hypergraphs for which a solution of equation \eqref{eq:1.3} is a minimizer of problem \eqref{eq:1.2} with $p=2$.
For this reason, $L^{H_n}$ is referred to as the hypergraph Laplacian in \cite{shi2024hypergraph1}.
In Figure \ref{fig:1}, a comparison of different graph and hypergraph models for a one-dimensional signal is given.
We observe that the spiky solution is suppressed by the hypergraph models.

The purpose of this paper is to give a theoretical study of equation \eqref{eq:1.3} with hyperedge \eqref{eq:1.1}.
We consider a more general discrete equation
\begin{align}\label{eq:1.4}
  \begin{cases}
    L^{H_n}_{p,\varepsilon_n}u_n(x_i)
    :=\sum\limits_{j=1}^n \eta\left(\frac{|x_i-x_j|}{\varepsilon_n}\right)
    \left(M^{k(p)\varepsilon_n} u_n(x_j)+M_{k(p)\varepsilon_n} u_n(x_j)-2u_n(x_i)\right)=0, &x_i\in \Omega_{n}\backslash\mathcal{O},\\
    u_n(x_i)=y_i, &x_i\in \mathcal{O}.
  \end{cases}
\end{align}
Here $\eta:[0,\infty)\rightarrow [0,\infty)$ is a nonnegative function with compact support $[0,1]$, $\varepsilon_n>0$ is a constant, $k(p)=\sqrt{c(p-2)}$ for $p\geq 2$ and $c>0$,
\begin{equation*}
  M^{r} u_n(x_i)=\max\limits_{x_j\in {B}(x_i,r)}u_n(x_j),\quad
  M_{r} u_n(x_i)=\min\limits_{x_j\in {B}(x_i,r)}u_n(x_j),
\end{equation*}
$B(x_i,r)\subset\Omega_n$ is a ball centered at $x_i$ with radius $r$.
The new operator $L^{H_n}_{p,\varepsilon_n}$ is a generalization of the hypergraph Laplacian $L^{H_n}$ and the graph Laplacian $L^{G_n}_2$.
If $\eta(s)\equiv 1$ for $|s|\leq 1$ and $k(p)\equiv1$, $L^{H_n}_{p,\varepsilon_n}$ coincides with $L^{H_n}$.
We recover $L^{G_n}_2$ from $L^{H_n}_{p,\varepsilon_n}$ by taking  $k(p)\equiv 0$.

We study the asymptotic behavior of the solution of equation \eqref{eq:1.4} in the setting when the number of labeled points $N$ is fixed and the number of data points $n\rightarrow\infty$.
This corresponds to the regime where the labeled data is limited and the sampling rate $\frac{N}{n}$ is arbitrarily low.
We show that the discrete operator $L^{H_n}_{p,\varepsilon_n}$ is consistent with a weighted $p$-Laplacian for smooth functions.
To address the Dirichlet boundary condition in \eqref{eq:1.4}, we establish the uniform  H\"{o}lder estimate for the solution of equation \eqref{eq:1.4} under the assumption $p>d$.
This is motivated by the fact that the $p$-Laplacian equation admits a solution in $W^{1,p}(\Omega)$, which is H\"{o}lder continuous if $p>d$.
Combining the consistency and the uniform  H\"{o}lder estimate, we prove in the viscosity solution framework that the continuum limit of equation \eqref{eq:1.4} is the $p$-Laplacian equation \eqref{eq:1.7}.
This suggests equation \eqref{eq:1.4} as a new discretization of the continuum $p$-Laplacian equation on point clouds.
It should be noticed that the solutions of equation \eqref{eq:1.4}, \eqref{eq:1.5}, and \eqref{eq:1.6} are quite different for finite $n$,
even although they share the same limit as $n\rightarrow\infty$.

By applying the result to equation \eqref{eq:1.3} with hyperedge \eqref{eq:1.1},
we deduce from $k(p)\equiv 1$ that $p=d+4$ and the assumption $p>d$ is fulfilled automatically.
Namely, the solution of equation \eqref{eq:1.3} converges to a H\"{o}lder continuous function as $n\rightarrow\infty$.
This partially explains the numerical experiments in \cite{shi2024hypergraph1}, as it prevents spikes better than the graph Laplacian.
We further compare the hypergraph equation \eqref{eq:1.4} with the graph Laplacian equation through numerical experiments on two high-dimensional data interpolation problems, namely image inpainting and semi-supervised learning. The experimental results demonstrate that the hypergraph equation outperforms the graph Laplacian equation at the extremely low labeling rates.

The definition of the hypergraph $p$-Laplacian is not unique. In addition to \eqref{eq:1.2}, several alternative definitions have been introduced for different applications \cite{saito2023generalizing,fazeny2024hypergraph,shi2025continuum}. The study of the continuum limits of hypergraph $p$-Laplacians is substantially more challenging and deserves further investigation. We refer the reader to \cite{shi2024hypergraph,shi2025continuum,weihs2025analysis} for some of the related work.

This paper is organized as follows.
In Section 2, we present assumptions and the main result of this paper.
The consistency between the hypergraph $p$-Laplacian operator $L^{H_n}_{p,\varepsilon_n}$ and the continuum $p$-Laplacian operator for smooth functions is given in Section 3. A nonlocal $p$-Laplacian operator is introduced as a bridge between the discrete and continuum settings.
In Section 4, we establish the H\"{o}lder estimate for the solution of the hypergraph $p$-Laplacian equation.
The main result, i.e., the continuum limit of equation \eqref{eq:1.4}, is then proven as a corollary of the consistency and the H\"{o}lder estimate.
Section 5 is devoted to numerical experiments. We conclude this paper in Section 6.

\begin{figure}[t]
\centering
\begin{subfigure}[t]{0.32\linewidth}
  \centering
  \includegraphics[width=\linewidth]{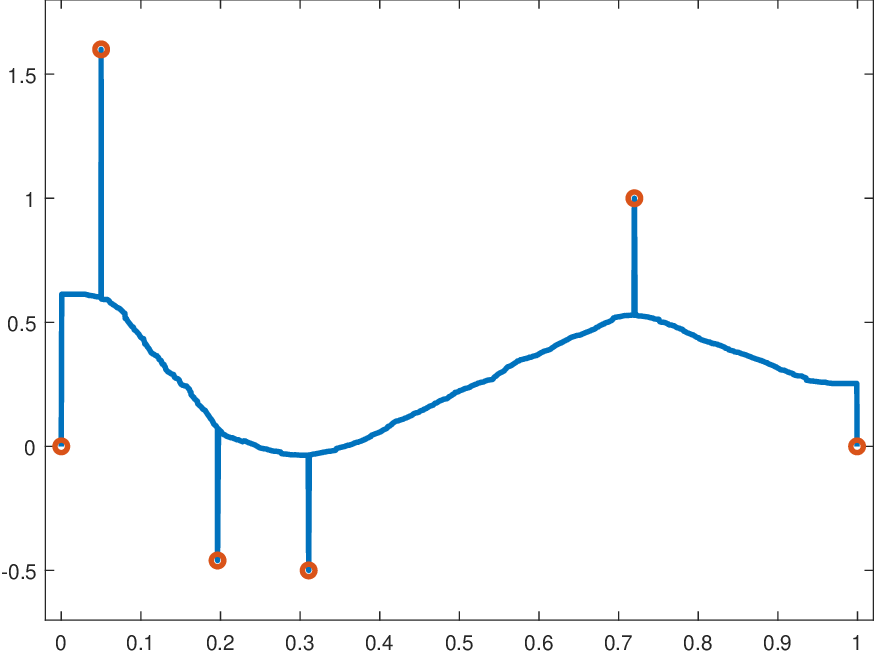}
  \caption{$L^{G_n}_2$}
  \label{fig:1-a}
\end{subfigure}\hfill
\begin{subfigure}[t]{0.32\linewidth}
  \centering
  \includegraphics[width=\linewidth]{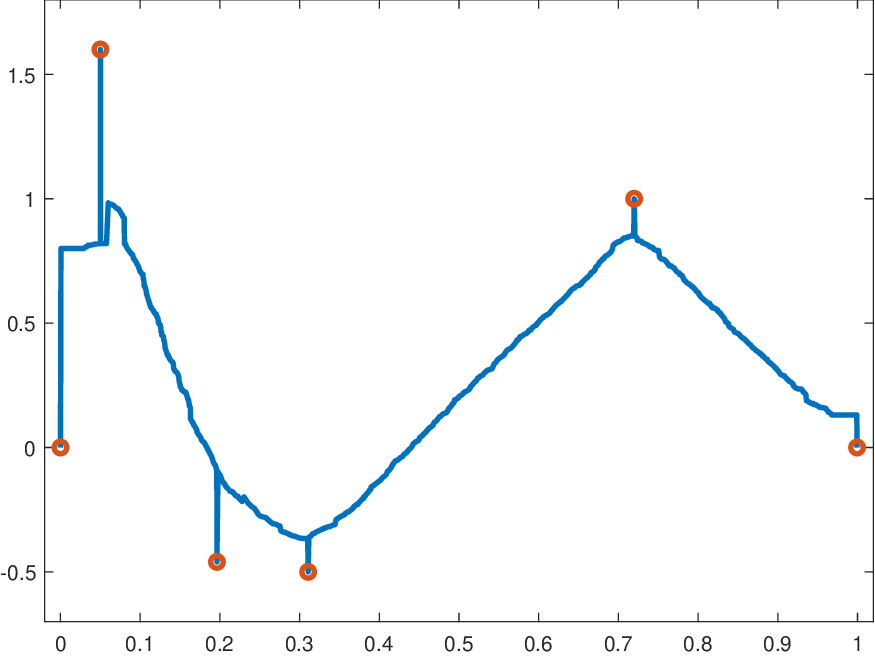}
  \caption{$L^{G_n}_4$}
  \label{fig:1-b}
\end{subfigure}\hfill
\begin{subfigure}[t]{0.32\linewidth}
  \centering
  \includegraphics[width=\linewidth]{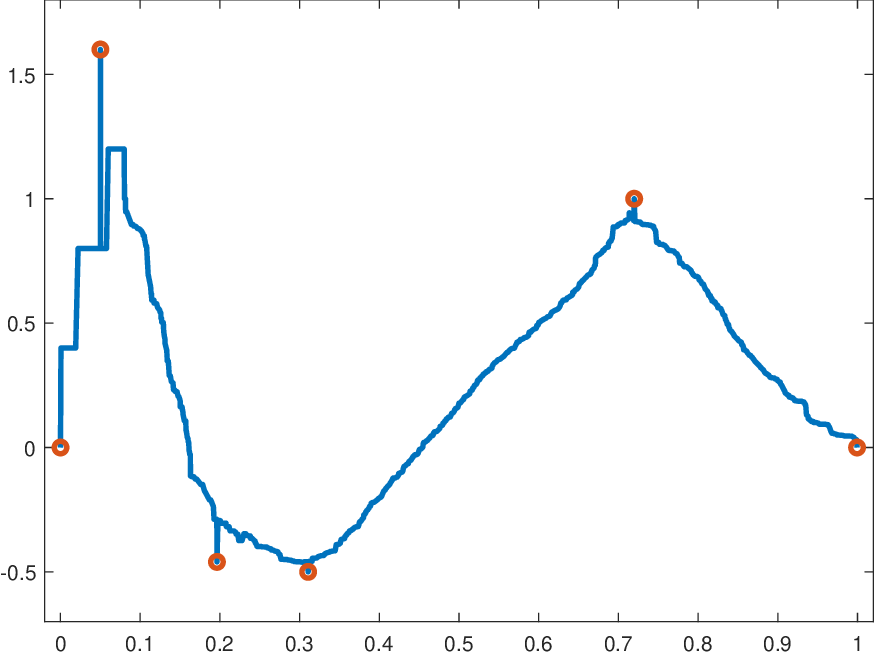}
  \caption{$L^{G_n}_{\infty}$}
  \label{fig:1-c}
\end{subfigure}

\vspace{0.6em}

\begin{subfigure}[t]{0.32\linewidth}
  \centering
  \includegraphics[width=\linewidth]{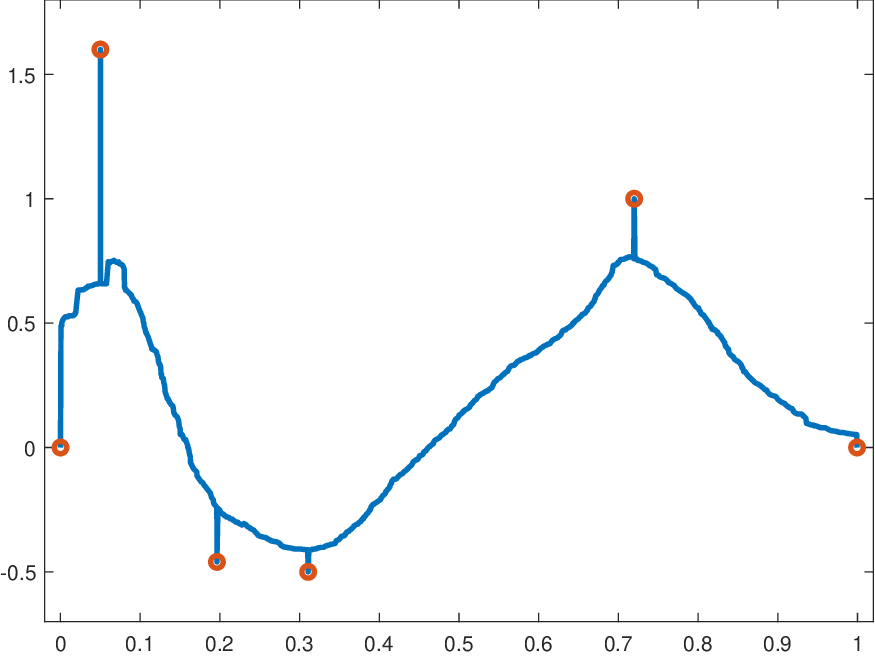}
  \caption{$\mathcal{L}^{G_n}_p$}
  \label{fig:1-d}
\end{subfigure}\hfill
\begin{subfigure}[t]{0.32\linewidth}
  \centering
  \includegraphics[width=\linewidth]{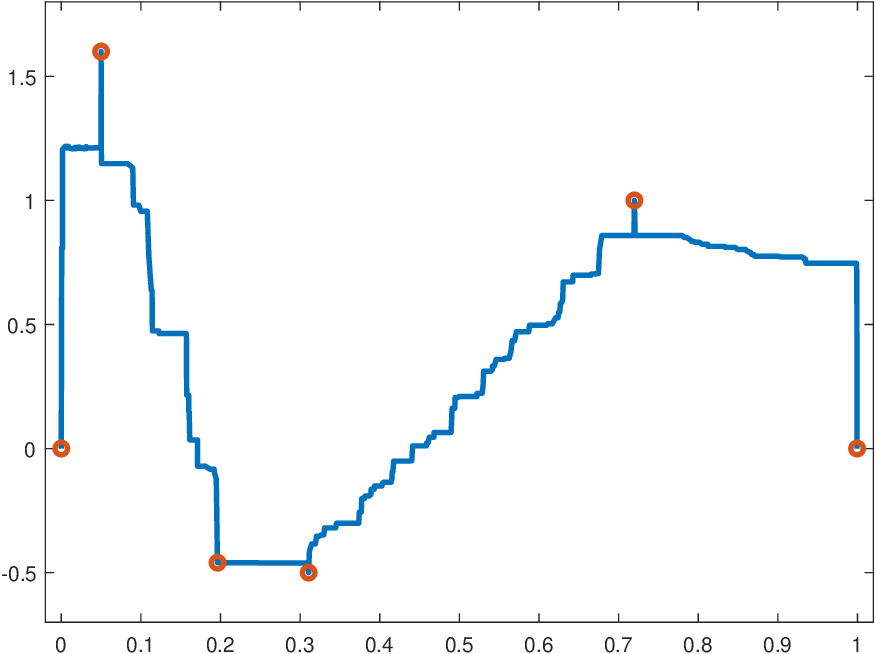}
  \caption{$F_2^{H_n}$}
  \label{fig:1-e}
\end{subfigure}\hfill
\begin{subfigure}[t]{0.32\linewidth}
  \centering
  \includegraphics[width=\linewidth]{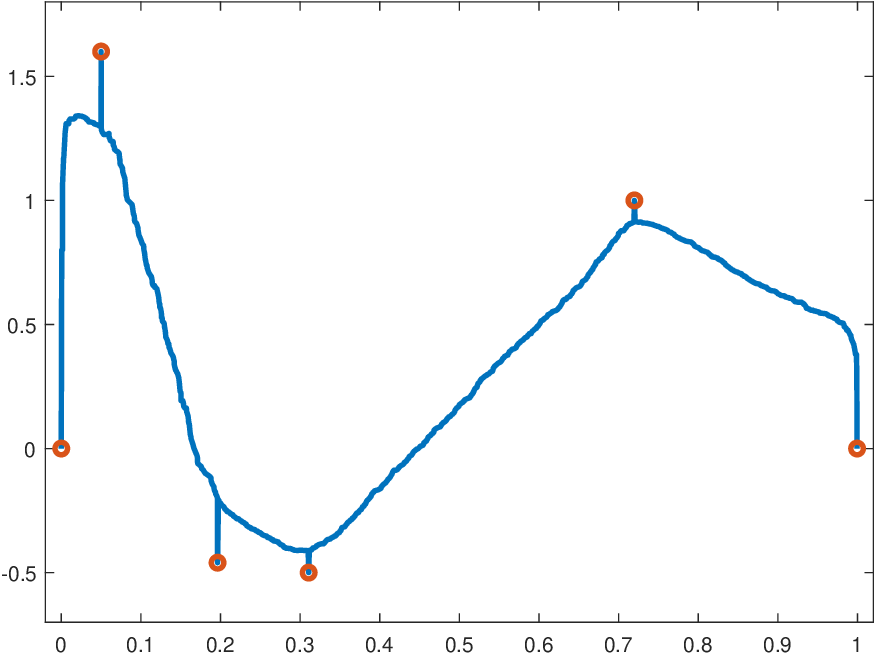}
  \caption{$L^{H_n}$}
  \label{fig:1-f}
\end{subfigure}

\caption{Interpolation results of the graph and hypergraph $p$-Laplacians for a one-dimensional signal with $n=1280$ points (drawn uniformly from the interval $[0,1]$) and $N=6$ labels (denoted by red circles). For each point, we select its nearest 72 points to construct edges or a hyperedge. The weights for the graph and the hypergraph are set to be 1.}
\label{fig:1}
\end{figure}

\section{Assumptions and the main result}

Let $\Omega\subset\mathbb{R}^d$ ($d\geq 2$) be a bounded, connected, and open domain with $C^1$ boundary $\partial\Omega$.
The unit outward normal on $\partial\Omega$ is denoted by $\vec{n}$.
We are given a set of points $\Omega_n=\{x_i\}_{i=1}^n\subset \Omega$ and a function $f$ defined on a subset $\mathcal{O}=\{x_i\}_{i=1}^N$ $(N<n)$. That is,
\begin{equation*}
  f:\mathcal{O}\rightarrow\R,\quad f(x_i)=y_i \mbox{ for } x_i\in \mathcal{O}.
\end{equation*}
Assume that $\Omega_n\backslash\mathcal{O}=\{x_i\}_{i=N+1}^n$ are independently and identically distributed random samples of a probability measure $\mu$ on $\Omega$ and $\mu$ has a Lebesgue density $\rho\in C^2(\overline{\Omega})$ with $\rho>0$ on $\overline{\Omega}$.
Without loss of generality, we also assume that the observed data $\mathcal{O}$ are far away from the boundary $\partial\Omega$.

Let $\eta: [0,\infty)\rightarrow [0,\infty)$ be a nonincreasing Lipschitz function with $\eta(0)\leq 2$, $\eta(1)\geq 1$, and
\begin{equation*}
  \eta_{\varepsilon_n}(s)=\frac{1}{\varepsilon_n^d}\eta\left(\frac{s}{\varepsilon_n}\right),
  \quad \eta^\chi_{\varepsilon_n}(s)=\chi_{\{|s|<\varepsilon_n\}}\eta_{\varepsilon_n}(s),
\end{equation*}
where $\chi_{\{|s|<\varepsilon_n\}}$ is the characteristic function of the set $\{|s|<\varepsilon_n\}$.
For $u_n\in \R^n$ and $x_i\in\Omega_{n}$, we define
\begin{equation*}
  M^{\varepsilon_n} u_n(x_i)=\max\limits_{x_j\in \overline{B}(x_i,\varepsilon_n)\cap{\Omega_n}}u_n(x_j),\quad
  M_{\varepsilon_n} u_n(x_i)=\min\limits_{x_j\in \overline{B}(x_i,\varepsilon_n)\cap{\Omega_n}}u_n(x_j),
\end{equation*}
where $B(x_i,\varepsilon_n)$ is a ball centered at $x_i$ with radius $\varepsilon_n$ and $\overline{B}(x_i,\varepsilon_n)$ is the closure of $B(x_i,\varepsilon_n)$.

The hypergraph $p$-Laplacian equation we consider for semi-supervised learning reads
\begin{align}\label{eq:2.1}
  \begin{cases}
    L^{H_n}_{p,\varepsilon_n}u_n(x_i)\\
    :=\frac{1}{2n\varepsilon_n^{2}}\sum\limits_{j=1}^n \eta^\chi_{\varepsilon_n}(|x_i-x_j|)\left(M^{k(p)\varepsilon_n} u_n(x_j)+M_{k(p)\varepsilon_n} u_n(x_j)-2u_n(x_i)\right)=0,\quad &x_i\in \Omega_{n}\backslash\mathcal{O},\\
    u_n(x_i)=y_i,\quad &x_i\in \mathcal{O},
  \end{cases}
\end{align}
where $p\geq 2$, $\frac{1}{2n\varepsilon_n^{2}}$ is the rescaling parameter,
\begin{equation}\label{eq:p}
  k(p)=\sqrt{\frac{(p-2)\sigma_{\eta,1}}{\sigma_{\eta,2}}},
\end{equation}
\begin{equation}\label{eq:eta}
  \sigma_{\eta,1}=\frac{1}{2}\int_{B(0,1)}\eta(|y|)|y_1|^2dy,
  \quad \sigma_{\eta,2}=\frac{1}{2}\int_{B(0,1)}\eta(|y|)dy,
\end{equation}
and $y_1$ is the first coordinate of $y$.
As mentioned in the introduction, $L^{H_n}_{p,\varepsilon_n}$ is a generalization of the hypergraph Laplacian.
We formally regard $L^{H_n}_{p,\varepsilon_n}$ as a discrete operator defined on the $\varepsilon_n$-ball hypergraph $H_n=\{V_n, E_n\}$ with $V_n=\Omega_n$ and
\begin{equation*}
  E_n=\{e_k\}_{k=1}^n,\quad e_k=\{x_j\in\Omega_n: |x_k-x_j|\leq \varepsilon_n\}.
\end{equation*}

The role of $k(p)$ will be clarified in Section \ref{se:3.1}, where we verify
by Taylor's expansion and optimal transportation that $L^{H_n}_{p,\varepsilon_n}$ is pointwisely consistent with a continuum weighted and normalized $(\frac{k(p)^2\sigma_{\eta,2}}{\sigma_{\eta,1}}+2)$-Laplacian operator for smooth functions as $\varepsilon_n\rightarrow 0$.
If $p=2$, i.e., $k(p)=0$, $L^{H_n}_{p,\varepsilon_n}$ is the well-known graph Laplacian.
For notational convenience, we sometimes use $k$ instead of $k(p)$ in the following, but keep in mind that it is always defined by \eqref{eq:p}.

The unique solvability of equation \eqref{eq:2.1} and the comparison principle have been studied in \cite{shi2024hypergraph1} when $k=1$ and $\eta$ is a constant function. The proofs apply to general $k\geq 0$ and to $\eta$ satisfying the assumptions above. We present the results in the following.
\begin{theorem}\label{th:hypergraph_existence}
  Let $p\geq 2$, $n>N>0$ be fixed, and $\varepsilon_n$ be large such that the hypergraph $H_n$ is connected. Then equation \eqref{eq:2.1} admits a unique solution $u_n\in\R^n$.
\end{theorem}

\begin{proposition}\label{pr:comparison_principle_hypergraph}
Let $p\geq 2$,  $n>N>0$ be fixed, and $\varepsilon_n$ be large such that the hypergraph $H_n$ is connected.
  Assume that $u_n,v_n: \Omega_n\rightarrow\R$ and
  \begin{equation*}
    L^{H_n}_{p,\varepsilon_n}u_n(x_i)\geq L^{H_n}_{p,\varepsilon_n}v_n(x_i), \quad \mbox{for any } x_i\in\Omega_{n}\backslash\mathcal{O}.
  \end{equation*}
  Then
  \begin{equation*}
    \max_{\Omega_n}(u_n-v_n)=\max_{\mathcal{O}}(u_n-v_n).
  \end{equation*}
  In particular, if $u_n\leq v_n$ on $\mathcal{O}$, then $u_n\leq v_n$ on $\Omega_n$.
\end{proposition}

We now consider the continuum limit of equation \eqref{eq:2.1}, which is the following weighted $p$-Laplacian equation with the mixed Dirichlet and Neumann boundary conditions,
\begin{align}\label{eq:2.2}
\begin{cases}
  \Delta^\rho_p u:=\mbox{div}\left(\rho^2|\nabla u|^{p-2}\nabla u\right)=0, & \mbox{in } \Omega\backslash\mathcal{O}, \\
  u=f, & \mbox{on }  \mathcal{O}, \\
  \frac{\partial u}{\partial \vec{n}}=0, & \mbox{on } \partial\Omega.
\end{cases}
\end{align}
Notice that $\mathcal{O}$ is a set of finite points. The Dirichlet boundary condition is well-defined only when the solution is continuous.
According to the classical theory of the $p$-Laplacian equation, it requires that $p>d$.

Since the hypergraph $p$-Laplacian equation \eqref{eq:2.1} does not have a weak formulation,
it is convenient to consider the problem in the viscosity solution framework.
Let $USC(\overline{\Omega})$ and $LSC(\overline{\Omega})$ be the set of upper semicontinuous functions and the set of lower semicontinuous functions on $\overline{\Omega}$, respectively.
We recall the definition of the viscosity solution for equation \eqref{eq:2.2} \cite{juutinen2001equivalence}.
\begin{definition}
  A function $u\in USC(\overline{\Omega})$ is a viscosity subsolution of equation
  \begin{align}\label{eq:2.3}
    \begin{cases}
      -\Delta^\rho_p u= 0,\quad \mbox{in } \Omega\backslash\mathcal{O},\\
      \frac{\partial u}{\partial\vec{n}}=0, \quad \mbox{on } \partial\Omega,
    \end{cases}
  \end{align}
  if for any $\varphi\in C^2(\overline{\Omega})$  and $x_0\in \overline{\Omega}\backslash\mathcal{O}$  such that
  $u-\varphi$ attains a strict maximum at $x_0$ and $\nabla\varphi(x_0)\neq 0$, one has
  \begin{equation*}
    -\Delta^\rho_p \varphi(x_0)\leq 0,\quad \mbox{if } x_0\in\Omega,
  \end{equation*}
  and
  \begin{equation*}
  \min\left\{-\Delta^\rho_p \varphi(x_0), \frac{\partial\varphi}{\partial\vec{n}}(x_0)\right\}\leq 0, \quad \mbox{if } x_0\in\partial\Omega.
  \end{equation*}
  A function $v\in LSC(\overline{\Omega})$ is a viscosity supersolution of the equation if $-v$ is a subsolution of \eqref{eq:2.3}.
  A function $u\in C(\overline{\Omega})$ is a viscosity solution of \eqref{eq:2.3} if it is both a viscosity subsolution and a viscosity supersolution.
\end{definition}

The unique solvability of equation \eqref{eq:2.2} is stated as follows. We omit the proof here, since the proof for the Dirichlet problem \cite{calder2018game} can be adapted for equation \eqref{eq:2.2} directly.

\begin{theorem}\label{th:existence}
  Let $p>d$. Equation \eqref{eq:2.2} admits a unique viscosity solution $u\in C(\overline{\Omega})$ in the sense that $u$ is a viscosity solution of equation \eqref{eq:2.3} and $u=f$ on $\mathcal{O}$.
\end{theorem}

The main result of this paper is the convergence for the solution of the hypergraph $p$-Laplacian equation \eqref{eq:2.1} to the solution of the continuum $p$-Laplacian equation \eqref{eq:2.2}.

\begin{theorem}\label{th:convergence}
  Let $p>d$,
      \begin{align}\label{delta}
    \delta_n=
      \begin{cases}
        \frac{(\ln n)^{3/4}}{\sqrt{n}}, & \mbox{if } d=2, \\
        \frac{(\ln n)^{1/d}}{n^{1/d}}, & \mbox{if  } d\geq 3,
      \end{cases}
    \end{align}
   and $\sqrt{\delta_n}\ll \varepsilon_n\ll 1$.
  If $u_n$ is the solution of equation \eqref{eq:2.1} and $u$ is the solution of equation \eqref{eq:2.2}, then  with probability one,
  \begin{equation*}
    u_n \rightarrow u, \quad \mbox{uniformly in } \Omega,
  \end{equation*}
  as $n\rightarrow \infty$.
\end{theorem}

The assumption $\varepsilon_n\gg \sqrt{\delta_n}$ in Theorem \ref{th:convergence} will become clear when we translate the discrete operator $L^{H_n}_{p,\varepsilon_n}$ into a nonlocal operator via the optimal transposition and estimate the error it introduced.
It also implies the connectivity of the hypergraph $H_n$.
In fact, by \cite{penrose2003random,garcia2016continuum}, the $\varepsilon_n$-ball graph is connected
with probability one.
Namely, for any $x_i,x_j\in\Omega_n$, there exist a sequence of vertices $x_{k_1}=x_i, x_{k_2},\cdots, x_{k_{s-1}},
x_{k_s}=x_j$, such that $|x_{k_l}-x_{k_{l-1}}|\leq \varepsilon_n$ for any $2\leq l\leq s$.
This indicates that hyperedge $e_k$ has cardinality $|e_k|\geq 2$ for any $k=1,2\cdots,n$ and hypergraph $H_{n}$ is connected.

The assumption $p>d$ in Theorem \ref{th:convergence} can be fulfilled with a flat kernel $\eta$ or a large $k$.
We conclude this section with two examples.

\begin{remark}\label{re:1}
We recall an integration formula from \cite[Lemma 2.1]{ishii2010class}
\begin{equation}\label{eq:remark1}
  \int_{B(0,1)}\eta(|x|^2)|x_1|^{2p_1-1}\cdots|x_d|^{2p_d-1}dx
  =\frac{\Gamma(p_1)\cdots\Gamma(p_d)}{\Gamma(p_1+\cdots+p_d)}
  \int_0^1\eta(t)t^{p_1+\cdots +p_d-1}dt,
\end{equation}
where the constant $p_i\geq \frac{1}{2}$ for $1\leq i\leq d$ and $\Gamma(t)=\int_{0}^{\infty}e^{-x}x^{t-1}dt$ is the Gamma function.
\begin{itemize}
  \item If $k(p)=1$ and $\eta\equiv 1$, $L^{H_n}_{p,\varepsilon_n}$ is exactly the hypergraph operator
      $L^{H_n}$ defined in \eqref{eq:1.3} with hyperedge \eqref{eq:1.1}.
In this case it follows from \eqref{eq:remark1} that $\frac{\sigma_{\eta,2}}{\sigma_{\eta,1}}=d+2$ and  $p=\frac{\sigma_{\eta,2}}{\sigma_{\eta,1}}+2=d+4$.
  \item If $k(p)=l$ for a constant $l\geq 0$ and $\eta(t)=e^{-\frac{t^2}{\sigma^2}}$, which is the commonly used Gaussian kernel, we obtain from \eqref{eq:remark1} that
  \begin{equation*}
    \frac{\sigma_{\eta,2}}{\sigma_{\eta,1}}=d\frac{\int_{0}^{1}e^{-\frac{t}{\sigma^2}}t^{\frac{d}{2}-1}dt}
    {\int_{0}^{1}e^{-\frac{t}{\sigma^2}}t^{\frac{d}{2}}dt}
    =\frac{\frac{2}{\sigma^2}\int_{0}^{1}e^{-\frac{t}{\sigma^2}}t^{\frac{d}{2}}dt+2e^{-\frac{1}{\sigma^2}}}
    {\int_{0}^{1}e^{-\frac{t}{\sigma^2}}t^{\frac{d}{2}}dt}>\frac{2}{\sigma^2}.
  \end{equation*}
  Consequently, $p=\frac{\sigma_{\eta,2}}{\sigma_{\eta,1}}l^2+2>\frac{2l^2}{\sigma^2}+2$. To satisfy the assumption $p>d$, it suffices that $l\geq \sqrt{\frac{(d-2)\sigma^2}{2}}$.
\end{itemize}
\end{remark}

\section{Consistency for the hypergraph $p$-Laplacian operator}\label{se:3.1}

For $\varphi\in C^2(\mathbb{R}^d)$, let
\begin{equation*}
  \Delta^\rho \varphi=\frac{1}{\rho}\mbox{div}(\rho^2\nabla \varphi)\quad \mbox{and} \quad
  \Delta_\infty^\rho \varphi=\rho\Delta_\infty \varphi
  =\rho\frac{1}{|\nabla \varphi|^2}\varphi_{x_ix_j}\varphi_{x_i}\varphi_{x_j}
\end{equation*}
be the weighted Laplacian and weighted $\infty$-Laplacian.
Here and in the following we utilize the Einstein summation convention.
Define
\begin{align*}
  L_p\varphi&=\sigma_{\eta,1}\Delta^\rho \varphi +k^2\sigma_{\eta,2}\Delta_\infty^\rho \varphi
  =\sigma_{\eta,1}\left(\Delta^\rho \varphi+\frac{k^2\sigma_{\eta,2}}{\sigma_{\eta,1}}\Delta_\infty^\rho \varphi\right)
  =\sigma_{\eta,1}\left(\Delta^\rho \varphi+(p-2)\Delta_\infty^\rho \varphi\right)\\
  &=\frac{\sigma_{\eta,1}}{\rho}|\nabla \varphi|^{2-p}\mbox{div}\left(\rho^2|\nabla \varphi|^{p-2}\nabla \varphi\right)
  =\frac{\sigma_{\eta,1}}{\rho}|\nabla \varphi|^{2-p}\Delta^\rho_p\varphi,
\end{align*}
which is a weighted and normalized $p$-Laplacian operator.
In this section, we show the consistency between the hypergraph $p$-Laplacian operator $L^{H_n}_{p,\varepsilon_n}$ and the continuum $p$-Laplacian operator $L_{p}$ for smooth functions.
A nonlocal $p$-Laplacian operator is needed as a bridge between the discrete and continuum settings.

Let $V$ and $D$ be two subsets of $\R^d$. We use the notation
\begin{equation*}
  V_D=V\backslash D \quad\mbox{and}\quad (V_D)^\varepsilon=\{x\in V: \mbox{dist}(x,D)>\varepsilon\},
\end{equation*}
for a small constant $\varepsilon>0$.
For $\varphi\in C(\overline{\Omega})$, define
\begin{equation*}
  M^{\varepsilon} \varphi(x)=\max\limits_{y\in \overline{B}(x,\varepsilon)\cap\overline{\Omega}}\varphi(y),\quad
  M_{\varepsilon} \varphi(x)=\min\limits_{y\in \overline{B}(x,\varepsilon)\cap\overline{\Omega}}\varphi(y),\quad
  x\in\overline{\Omega}.
\end{equation*}
The nonlocal $p$-Laplacian operator with rescaling parameter $\frac{1}{2\varepsilon^{2}}$ and weight $\rho$ reads
\begin{equation}\label{eq:3.1}
  L_{p,\varepsilon}\varphi(x)=\frac{1}{2\varepsilon^{2}}\int_{\Omega}\eta^\chi_\varepsilon(|x-y|)
  \left(M^{k\varepsilon} \varphi(y)+M_{k\varepsilon} \varphi(y)-2\varphi(x)\right)\rho(y)dy,
  \quad x\in (\overline{\Omega}_\mathcal{O})^{(k+1)\varepsilon},
\end{equation}
where $k=\sqrt{\frac{(p-2)\sigma_{\eta,1}}{\sigma_{\eta,2}}}$, as in the discrete case \eqref{eq:p}.

The consistency between $L_{p,\varepsilon}$ and $L_p$ takes different forms in the interior of $\Omega$ and near its boundary $\partial\Omega$.
The consistency in the interior of $\Omega$ is as follows.

\begin{lemma}\label{le:consistency_interior}
Let $\varphi\in C^3(\overline{\Omega})$, $\varepsilon>0$ be sufficiently small, and $x\in (\overline{\Omega}_{\mathcal{O}\cap\partial\Omega})^{(k+1)\varepsilon}$.
If $\nabla \varphi(x)\neq 0$,
then
\begin{equation*}
  |L_{p,\varepsilon}\varphi(x)-L_p\varphi(x)|\leq C\varepsilon,
\end{equation*}
where the constant $C$ depends only on $\|\varphi\|_{C^3(\overline{\Omega})}$.
\end{lemma}
\begin{proof}
  By a change of variables $y=x+\varepsilon\hat{y}$,
  \begin{align*}
    L_{p,\varepsilon}\varphi(x)&=\frac{1}{2\varepsilon^2}\int_{B(0,1)} \eta(|\hat{y}|)h(\hat{y})\rho(x+\varepsilon\hat{y})d\hat{y}\\
    &=\frac{1}{2\varepsilon^2}\int_{B(0,1)} \eta(|\hat{y}|)h(\hat{y})\left(\rho(x+\varepsilon\hat{y})-\rho(x)\right)d\hat{y}
    +\frac{\rho(x)}{2\varepsilon^2}\int_{B(0,1)} \eta(|\hat{y}|)h(\hat{y})d\hat{y}\\
    &=:I_1+\rho(x)I_2,
  \end{align*}
where by a change of variables $z=x+\varepsilon\hat{y}+k\varepsilon\hat{z}$  and Taylor's expansion,
\begin{align}\label{eq:consistency_interior:3}
\begin{split}
  h(\hat{y})
  =&M^{k\varepsilon} \varphi(x+\varepsilon\hat{y})+M_{k\varepsilon} \varphi(x+\varepsilon\hat{y})-2\varphi(x) \\
  =&\sup_{z\in B(x+\varepsilon\hat{y},k\varepsilon)}\varphi(z)
  +\inf_{z\in B(x+\varepsilon\hat{y},k\varepsilon)}\varphi(z) -2\varphi(x) \\
  =&\sup\limits_{\hat{z}\in B(0,1)}\varphi\left(x+\varepsilon(\hat{y}+k\hat{z})\right)
  +\inf\limits_{\hat{z}\in B(0,1)}\varphi\left(x+\varepsilon(\hat{y}+k\hat{z})\right) -2 \varphi(x) \\
  =&\sup\limits_{\hat{z}\in B(0,1)}\left(\varepsilon\nabla \varphi(x)\cdot(\hat{y}+k\hat{z}) +\frac{\varepsilon^2}{2}\varphi_{x_ix_j}(x)(\hat{y}_i+k\hat{z}_i)(\hat{y}_j+k\hat{z}_j))\right)\\
  &+\inf\limits_{\hat{z}\in B(0,1)}\left(\varepsilon\nabla \varphi(x)\cdot(\hat{y}+k\hat{z}) +\frac{\varepsilon^2}{2}\varphi_{x_ix_j}(x)(\hat{y}_i+k\hat{z}_i)(\hat{y}_j+k\hat{z}_j)\right)+O(\varepsilon^3).
\end{split}
\end{align}
Notice that $\nabla \varphi(x)\neq 0$ and $\varepsilon$ is sufficiently small.
The supremum and infimum on the right-hand side of $h(\hat{y})$ are attained at $\hat{z}=\frac{\nabla \varphi(x)}{|\nabla \varphi(x)|}$ and $\hat{z}=-\frac{\nabla \varphi(x)}{|\nabla \varphi(x)|}$, respectively.
Consequently,
\begin{align*}
  h(\hat{y})
  &=2\varepsilon\nabla \varphi(x)\cdot\hat{y}
  +\varepsilon^2\left(\varphi_{x_ix_j}(x)\hat{y}_i\hat{y}_j
  +k^2\Delta_\infty \varphi(x)\right)  +O(\varepsilon^3),
\end{align*}
from which we deduce that
\begin{align*}
  I_1&=\frac{1}{2\varepsilon}\int_{B(0,1)} \eta(|\hat{y}|)h(\hat{y})
  (\nabla\rho(x)\cdot\hat{y}+O(\varepsilon)) d\hat{y}
  =\int_{B(0,1)}\eta(|\hat{y}|) \nabla \varphi(x)\cdot\hat{y}\nabla\rho(x)\cdot\hat{y} d\hat{y}+O(\varepsilon)\\
  &=2\sigma_{\eta,1} \nabla \varphi(x)\cdot\nabla \rho(x)+O(\varepsilon),
\end{align*}
and
\begin{align*}
  I_2
  =&\frac{1}{\varepsilon}\int_{B(0,1)}\eta(|\hat{y}|) \nabla \varphi(x)\cdot\hat{y}d\hat{y}
  +\frac{1}{2}\int_{B(0,1)} \eta(|\hat{y}|)\varphi_{x_ix_j}(x)\hat{y}_i\hat{y}_jd\hat{y}
  +\frac{k^2}{2}\int_{B(0,1)} \eta(|\hat{y}|)\Delta_\infty \varphi(x) d\hat{y}+O(\varepsilon)  \\
  =&\sigma_{\eta,1}\Delta\varphi(x)+ k^2\sigma_{\eta,2}\Delta_\infty \varphi(x)+O(\varepsilon),
\end{align*}
where $\sigma_{\eta,1}$ and $\sigma_{\eta,2}$ are defined in \eqref{eq:eta}.
Combining all results yields
\begin{align*}
  L_{p,\varepsilon}\varphi(x)
  &=I_1+\rho(x)I_2\\
  &=2\sigma_{\eta,1} \nabla \varphi(x)\cdot\nabla \rho(x)
  +\sigma_{\eta,1}\rho(x)\Delta\varphi(x)
  +k^2\sigma_{\eta,2}\rho(x)\Delta_\infty \varphi(x)+O(\varepsilon)\\
  &=\sigma_{\eta,1}\Delta^\rho \varphi(x)+k^2\sigma_{\eta,2}\Delta_\infty^\rho \varphi(x)+O(\varepsilon)\\
  &=L_p \varphi(x)+O(\varepsilon).
\end{align*}
This completes the proof.
\end{proof}


The consistency near the boundary $\partial\Omega$ is as follows.
\begin{lemma}\label{le:consistency_boundary}
Let $x\in\partial\Omega$ and $x_\varepsilon\in \overline{\Omega}$ such that $x_\varepsilon\rightarrow x$ as $\varepsilon\rightarrow 0$.
If $\varphi\in C^3(\overline{\Omega})$, $\nabla \varphi(x)\neq 0$, and $\varepsilon_0$ is sufficiently small,
then
\begin{equation}\label{eq:le:consistency_boundary:1}
-L_{p,\varepsilon}\varphi(x_\varepsilon)\leq 0 \mbox{ for any } \varepsilon\leq \varepsilon_0
~~\Longrightarrow~~
\min\left\{-L_p\varphi(x),\frac{\partial\varphi}{\partial\vec{n}}(x)\right\}\leq 0,
\end{equation}
and
\begin{equation}\label{eq:le:consistency_boundary:2}
-L_{p,\varepsilon}\varphi(x_\varepsilon)\geq 0  \mbox{ for any } \varepsilon\leq \varepsilon_0
~~\Longrightarrow~~
\max\left\{-L_p\varphi(x),\frac{\partial\varphi}{\partial\vec{n}}(x)\right\}\geq 0.
\end{equation}
\end{lemma}
\begin{proof}
Let us first prove \eqref{eq:le:consistency_boundary:1}.
  If $\nabla\varphi(x)$ towards inward from $\Omega$ or $\nabla\varphi(x)\perp \vec{n}(x)$, then $\frac{\partial\varphi}{\partial\vec{n}}(x)\leq0$ and the conclusion follows.
  Assume in the following that $\nabla\varphi(x)$ points outward from $\Omega$.
For sufficiently small $\varepsilon$, we also have that $\nabla\varphi(x_\varepsilon)$ points outward from $\Omega$ and $\nabla \varphi(x_\varepsilon)\neq 0$.

  By a change of variables $y=x_\varepsilon+\varepsilon\hat{y}$, $z=x_\varepsilon+\varepsilon\hat{y}+k\varepsilon\hat{z}$,  and Taylor's expansion,
  \begin{equation*}
    L_{p,\varepsilon}\varphi(x_\varepsilon)
    =\frac{1}{2\varepsilon^2}\int_{B(0,1)\atop x_\varepsilon+\varepsilon\hat{y}\in\Omega}\eta(|\hat{y}|) h(\hat{y})\rho(x_\varepsilon+\varepsilon\hat{y})d\hat{y},
  \end{equation*}
where
\begin{align*}
  h(\hat{y})&=\sup\limits_{\hat{z}\in B(0,1)\atop x_\varepsilon+\varepsilon(\hat{y}+k\hat{z})\in\Omega}\left(\varepsilon\nabla \varphi(x_\varepsilon)\cdot(\hat{y}+k\hat{z}) +\frac{\varepsilon^2}{2}\varphi_{x_ix_j}(x_\varepsilon)(\hat{y}_i+k\hat{z}_i)(\hat{y}_j+k\hat{z}_j)\right)\\
  &\qquad+\inf\limits_{\hat{z}\in B(0,1)\atop x_\varepsilon+\varepsilon(\hat{y}+k\hat{z})\in\Omega}\left(\varepsilon\nabla \varphi(x_\varepsilon)\cdot(\hat{y}+k\hat{z}) +\frac{\varepsilon^2}{2}\varphi_{x_ix_j}(x_\varepsilon)(\hat{y}_i+k\hat{z}_i)(\hat{y}_j+k\hat{z}_j)\right)
  +O(\varepsilon^3).
\end{align*}
Since
$x_\varepsilon+\varepsilon\hat{y}\in\Omega$ and $\nabla\varphi(x_\varepsilon)$ points outward from $\Omega$, we see that $x_\varepsilon+\varepsilon\left(\hat{y}-k\frac{\nabla\varphi(x_\varepsilon)}{|\nabla\varphi(x_\varepsilon)|}\right)\in\Omega$.
Consequently, the infimum on the right-hand side of $h(\hat{y})$ is attained at $\hat{z}=-\frac{\nabla \varphi(x_\varepsilon)}{|\nabla \varphi(x_\varepsilon)|}$.
This implies that
\begin{align*}
  h(\hat{y})
&\leq\sup\limits_{\hat{z}\in B(0,1)}\left(\varepsilon\nabla \varphi(x_\varepsilon)\cdot(\hat{y}+k\hat{z}) +\frac{\varepsilon^2}{2}\varphi_{x_ix_j}(x_\varepsilon)(\hat{y}_i+k\hat{z}_i)(\hat{y}_j+k\hat{z}_j)\right)\\
  &\qquad+\inf\limits_{\hat{z}\in B(0,1)}\left(\varepsilon\nabla \varphi(x_\varepsilon)\cdot(\hat{y}+k\hat{z}) +\frac{\varepsilon^2}{2}\varphi_{x_ix_j}(x_\varepsilon)(\hat{y}_i+k\hat{z}_i)(\hat{y}_j+k\hat{z}_j)\right)
  +O(\varepsilon^3)\\
  &=2\varepsilon\nabla \varphi(x_\varepsilon)\cdot\hat{y}
  +\varepsilon^2\left(\varphi_{x_ix_j}(x_\varepsilon)\hat{y}_i\hat{y}_j
  +k^2\Delta_\infty \varphi(x_\varepsilon)\right)  +O(\varepsilon^3)=:l(\hat{y}).
\end{align*}
It implies
  \begin{align*}
    L_{p,\varepsilon}\varphi(x_\varepsilon)
    \leq\frac{1}{2\varepsilon^2}\int_{B(0,1)\atop x_\varepsilon+\varepsilon\hat{y}\in\Omega}\eta(|\hat{y}|) l(\hat{y})\rho(x_\varepsilon+\varepsilon\hat{y})d\hat{y}
    &=\frac{1}{2\varepsilon^2}
    \left(\int_{B(0,1)}-\int_{B(0,1)\atop x_\varepsilon+\varepsilon\hat{y}\notin\Omega}\right)
    \eta(|\hat{y}|) l(\hat{y})\rho(x_\varepsilon+\varepsilon\hat{y})d\hat{y}\\
    &=:I_1-I_2.
  \end{align*}
By repeating the proof of Lemma \ref{le:consistency_interior}, we obtain that
\begin{equation*}
  I_1= L_p \varphi(x_\varepsilon)+O(\varepsilon).
\end{equation*}

Now we claim that
\begin{equation}\label{eq:le:consistency_boundary:3}
  I_2
    =\frac{1}{\varepsilon}
    \int_{B(0,1)\atop x_\varepsilon+\varepsilon\hat{y}\notin\Omega}
    \eta(|\hat{y}|) \left(\nabla \varphi(x_\varepsilon)\cdot\hat{y}+O(\varepsilon)\right)\rho(x_\varepsilon+\varepsilon\hat{y})d\hat{y}
    \geq 0,
\end{equation}
for any $x_\varepsilon\in\overline{\Omega}$ when $\varepsilon$ is sufficiently small.
If $\mbox{dist}(x_\varepsilon,\partial\Omega)\geq \varepsilon$, \eqref{eq:le:consistency_boundary:3} is trivial.
On the other hand, if $\mbox{dist}(x_\varepsilon,\partial\Omega)< \varepsilon$, we shall prove that
\begin{equation*}
  I_3:=\int_{B(0,1)\atop x_\varepsilon+\varepsilon\hat{y}\notin\Omega}\nabla \varphi(x_\varepsilon)\cdot\hat{y}d\hat{y}>0,
\end{equation*}
which implies \eqref{eq:le:consistency_boundary:3}. This follows from the regularity of the boundary $\partial\Omega$.

Let $\bar{x}_\varepsilon$ be the orthogonal projection of $x_\varepsilon$ on the boundary $\partial\Omega$.
Assume without loss of generality that $\vec{n}(\bar{x}_\varepsilon)=e_d$.
Since $\partial\Omega$ is $C^1$, there exists a $C^1$ function $\psi:\mathbb{R}^{d-1}\rightarrow\mathbb{R}$ such that the domain $\Omega$ lies below the graph of $\psi$, i.e.,
\begin{equation*}
  \Omega\cap B(\bar{x}_\varepsilon,\delta)=\{y=(y',y_d)\in B(\bar{x}_\varepsilon,\delta): y_d<\psi(y')\},
\end{equation*}
for a small constant $\delta>0$.
Moreover,
 $\nabla \psi(\bar{x}_\varepsilon')=\nabla \psi(x_\varepsilon')=0$.
Now by Taylor's expansion,
\begin{equation*}
x_\varepsilon+\varepsilon\hat{y}\notin\Omega
\Longleftrightarrow
x_{\varepsilon,d}+\varepsilon\hat{y}_d\geq \psi (x_\varepsilon'+\varepsilon\hat{y}')
\Longleftrightarrow
\hat{y}_d\geq \frac{\psi(x_\varepsilon')-x_{\varepsilon,d}}{\varepsilon}+o(1).
\end{equation*}
Consequently,
by the symmetry of the integral on the variable $\hat{y}'$,
\begin{align*}
  I_3&=\int_{B(0,1)\atop x_\varepsilon+\varepsilon\hat{y}\notin\Omega}\nabla \varphi(x_\varepsilon)\cdot\hat{y}d\hat{y}
  =\int_{B(0,1)\atop \hat{y}_d\geq \frac{\psi(x_\varepsilon')-x_{\varepsilon,d}}{\varepsilon}+o(1)}
  \varphi_{x_d}(x_\varepsilon) \hat{y}_dd\hat{y}
  =C\varphi_{x_d}(x_\varepsilon)\int_{\frac{\psi(x_\varepsilon')-x_{\varepsilon,d}}{\varepsilon}+o(1)}^{1}\hat{y}_dd\hat{y}_d\\
  &=C\varphi_{x_d}(x_\varepsilon)\left(1-\left(\frac{\psi(x_\varepsilon')-x_{\varepsilon,d}}{\varepsilon}+o(1)\right)^2\right),
\end{align*}
for a constant $C>0$.
Since $\nabla\varphi(x_\varepsilon)$ points outward from $\Omega$, we have $\varphi_{x_d}(x_\varepsilon)>0$.
Besides, by the assumption $\mbox{dist}(x_\varepsilon,\partial\Omega)< \varepsilon$, we further have $\frac{\psi(x_\varepsilon')-x_{\varepsilon,d}}{\varepsilon}=\frac{\mbox{dist}(x_\varepsilon,\partial\Omega)}{\varepsilon}<1$.
Substituting them into $I_3$ yields that $I_3>0$.
This proves \eqref{eq:le:consistency_boundary:3}.

To complete the proof of  \eqref{eq:le:consistency_boundary:1}, we collect all the above results and pass to the limit $\varepsilon\rightarrow 0$. Namely,
\begin{equation}\label{eq:Lemma7}
  L_p \varphi(x)
  =\lim_{\varepsilon\rightarrow 0}L_p \varphi(x_\varepsilon)
  =\lim_{\varepsilon\rightarrow 0}I_1
  \geq\lim_{\varepsilon\rightarrow 0}L_{p,\varepsilon} \varphi(x_\varepsilon)+I_2
  \geq \lim_{\varepsilon\rightarrow 0}L_{p,\varepsilon} \varphi(x_\varepsilon)\geq 0.
\end{equation}

To prove \eqref{eq:le:consistency_boundary:2}, we only need to consider $-\varphi$
and utilize the fact that $L_{p,\varepsilon}(-\varphi)(x_\varepsilon)=-L_{p,\varepsilon}\varphi(x_\varepsilon)$.
\end{proof}

The consistency results of Lemma \ref{le:consistency_interior} and Lemma \ref{le:consistency_boundary} can be generalized to the hypergraph $p$-Laplacian operator $L^{H_n}_{p,\varepsilon_n}$ via the transportation map
\begin{equation*}
  T_n: \Omega\rightarrow\Omega_n,\quad T_{n \sharp}\mu=\mu_n,
\end{equation*}
where $\mu_n=\frac{1}{n}\sum_{i=1}^{n}\delta_{x_i}$ is the empirical measure on $\Omega_n$
and $T_{n \sharp}\mu$ is the push-forward measure of $\mu$ by $T_n$ defined as
\begin{equation*}
  T_{n \sharp}\mu(A)=\mu\left(T^{-1}_n(A)\right),\quad A\in\mathcal{B}(\Omega),
\end{equation*}
for the Borel $\sigma$-algebra $\mathcal{B}(\Omega)$.
It leads to a change of variables,
\begin{equation}\label{eq:change_variable}
  \frac{1}{n}\sum_{i=1}^{n}\varphi(x_i)=\int_{\Omega}\varphi(x)d\mu_n(x)
  =\int_{\Omega}\varphi(T_n(x))d\mu(x)=\int_{\Omega}\varphi(T_n(x))\rho(x)dx,
\end{equation}
for a function $\varphi:\Omega_n\rightarrow\mathbb{R}$.
The error incurred in translating the hypergraph operator to a nonlocal operator is bounded by the following lemma
\cite{trillos2015rate}.
\begin{lemma}
  Let $d\geq 2$ and $\{X_i\}_{i=1}^\infty$ be a sequence of independent random variables with distribution $\mu$ on $\Omega$. Then there exists a sequence of transportation maps $\{T_n\}_{n=1}^\infty$ from $\mu$ to $\mu_n$, such that
  \begin{align}\label{measure}
  \begin{split}
    c\leq\liminf_{n\rightarrow\infty}\frac{\|Id-T_n\|_{L^\infty(\Omega)}}{\delta_n}
    \leq\limsup_{n\rightarrow\infty}\frac{\|Id-T_n\|_{L^\infty(\Omega)}}{\delta_n}
    \leq C,
  \end{split}
  \end{align}
  almost surely, where $\delta_n$ is defined in \eqref{delta}.
\end{lemma}

The Lemma is stated in the random setting. As in \cite{garcia2016continuum}, we regard the data points as the realization of the random variables $\{X_i\}_{i=1}^\infty$ and consider our problem in the deterministic setting.

In the rest of this paper, we always assume that $\varepsilon_n\gg \sqrt{\delta_n}$.
It follows from \eqref{measure} that
\begin{equation}\label{eq:Tn}
  \lim_{n\rightarrow\infty}\frac{\|Id-T_n\|_{L^\infty(\Omega)}}{\varepsilon_n^2}=0.
\end{equation}
This is what we need in the following.

\begin{lemma}\label{le:consistency_hypergraph_interior}
  Let $\varphi\in C^3(\overline{\Omega})$, $\varepsilon_n>0$ be sufficiently small, and $x_i\in \Omega_n$ with $\mbox{dist}(x_i,\mathcal{O}\cap\partial\Omega)>(k+1)\varepsilon_n$.
  If $\nabla \varphi(x_i)\neq 0$, then
\begin{equation*}
  |L^{H_n}_{p,\varepsilon_n}\varphi(x_i)-L_p\varphi(x_i)|\leq C\varepsilon_n
  +\frac{C\|Id-T_n\|_{L^\infty(\Omega)}}{\varepsilon^{2}_n},
\end{equation*}
where the constant $C$ depends only on $\|\varphi\|_{C^3(\overline{\Omega})}$.

\end{lemma}
\begin{proof}
  By a change of variables \eqref{eq:change_variable},
\begin{align}\label{eq:3.2}
\begin{split}
  L^{H_n}_{p,\varepsilon_n}\varphi(x_i)
  &=\frac{1}{2\varepsilon_n^{2}}\int_{\Omega}\eta^\chi_{\varepsilon_n}(|x_i-T_n(y)|)
  \left(M^{k\varepsilon_n} \varphi(T_n(y))+M_{k\varepsilon_n} \varphi(T_n(y))-2\varphi(x_i)\right)\rho(y)dy\\
    &=L_{p,\varepsilon_n}\varphi(x_i)+I,
\end{split}
\end{align}
where
\begin{align*}
  I=
  &\frac{1}{2\varepsilon_n^{2}}\int_{\Omega}\eta^\chi_{\varepsilon_n}(|x_i-T_n(y)|)
  \left(M^{k\varepsilon_n} \varphi(T_n(y))+M_{k\varepsilon_n} \varphi(T_n(y))
  -M^{k\varepsilon_n} \varphi(y)-M_{k\varepsilon_n} \varphi(y)\right)\rho(y)dy\\
  &+
  \frac{1}{2\varepsilon_n^{2}}\int_{\Omega}
  \left(\eta^\chi_{\varepsilon_n}(|x_i-T_n(y)|)-\eta^\chi_{\varepsilon_n}(|x_i-y|)\right)
  \left(M^{k\varepsilon_n} \varphi(y)+M_{k\varepsilon_n} \varphi(y)-2\varphi(x_i)\right)\rho(y)dy\\
  =:&I_1+I_2.
\end{align*}
We estimate $I_1$ and $I_2$ in the following.

Let
\begin{equation*}
  \tilde{\varepsilon}_n=\varepsilon_n+\|Id-T_n\|_{L^\infty(\Omega)}.
\end{equation*}
It follows that
\begin{equation*}
  \frac{|x_i-T_n(y)|}{\varepsilon_n}\leq 1\Longrightarrow
  \frac{|x_i-y|}{\tilde{\varepsilon}_n}\leq 1, \quad \mbox{for any } y\in\Omega,
\end{equation*}
and
\begin{equation*}
  \int_{\Omega}\eta^\chi_{\varepsilon_n}(|x_i-T_n(y)|)dy
  \leq \frac{C}{\varepsilon_n^d}\int_{\Omega}\chi_{\{|x_i-T_n(y)|<\varepsilon_n\}}dy
  \leq \frac{C}{\varepsilon_n^d}\int_{\Omega}\chi_{\{|x_i-y|<\tilde{\varepsilon}_n\}}dy
  \leq C\left(\frac{\tilde{\varepsilon}_n}{\varepsilon_n}\right)^d.
\end{equation*}
We further have
\begin{equation*}
  |M^{k\varepsilon_n} \varphi(T_n(y))-M^{k\varepsilon_n} \varphi(y)|
  \leq C|T_n(y)-y|~~ \mbox{and }~~
  |M_{k\varepsilon_n} \varphi(T_n(y))-M_{k\varepsilon_n} \varphi(y)|
  \leq C|T_n(y)-y|.
\end{equation*}
Consequently,
\begin{equation*}
  |I_1|\leq C\left(\frac{\tilde{\varepsilon}_n}{\varepsilon_n}\right)^d
  \frac{\|Id-T_n\|_{L^\infty(\Omega)}}{\varepsilon_n^2}
  \leq C\frac{\|Id-T_n\|_{L^\infty(\Omega)}}{\varepsilon_n^2}.
\end{equation*}

To estimate $I_2$, we define a new constant
\begin{equation*}
  \hat{\varepsilon}_n=\varepsilon_n-\|Id-T_n\|_{L^\infty(\Omega)},
\end{equation*}
and see that
\begin{align*}
  |x_i-y|< \hat{\varepsilon}_n \Longrightarrow |x_i-y|<\varepsilon_n ~\text{ and }~ |x_i-T_n(y)|<\varepsilon_n, \\
  |x_i-y|> \tilde{\varepsilon}_n \Longrightarrow |x_i-y|>\varepsilon_n ~\text{ and }~ |x_i-T_n(y)|>\varepsilon_n,
\end{align*}
Now the integral in $I_2$ can be split into three parts, i.e.,
\begin{equation*}
  I_2=\int_\Omega=
  \int_{|x_i-y|> \tilde{\varepsilon}_n}
  +\int_{|x_i-y|< \hat{\varepsilon}_n}
  +\int_{\hat{\varepsilon}_n\leq|x_i-y|\leq \tilde{\varepsilon}_n}
  =:I_{2,1}+I_{2,2}+I_{2,3}.
\end{equation*}
By the definition, $I_{2,1}=0$.
For the second integral, we have, by the Lipschitz of $\eta$ and $\varphi$,
\begin{align*}
  |I_{2,2}|=&
  \frac{1}{2\varepsilon_n^{2}}\int_{|x_i-y|< \hat{\varepsilon}_n}
  \left|\eta_{\varepsilon_n}(|x_i-T_n(y)|)-\eta_{\varepsilon_n}(|x_i-y|)\right|
  \left|M^{k\varepsilon_n} \varphi(y)+M_{k\varepsilon_n} \varphi(y)-2\varphi(x_i)\right|\rho(y)dy\\
  &\leq \frac{C}{\varepsilon_n^{d+2}}\int_{|x_i-y|< \hat{\varepsilon}_n}
  \left|\frac{|x_i-T_n(y)|}{\varepsilon_n}-\frac{|x_i-y|}{\varepsilon_n}\right|
  \left(|x_i-y|+k\varepsilon_n\right)dy\\
  &\leq C\left(\frac{\hat{\varepsilon}_n}{\varepsilon_n}\right)^d
  \frac{\|Id-T_n\|_{L^\infty(\Omega)}}{\varepsilon_n^2}
  \leq C\frac{\|Id-T_n\|_{L^\infty(\Omega)}}{\varepsilon_n^2}.
\end{align*}
For the third integral, we have
\begin{align*}
  |I_{2,3}|
  &\leq \frac{C}{\varepsilon^{d+2}_n}\int_{\hat{\varepsilon}_n\leq|x_i-y|\leq \tilde{\varepsilon}_n}
  \left|M^{k\varepsilon_n} \varphi(y)+M_{k\varepsilon_n} \varphi(y)-2\varphi(x_i)\right|dy\\
 &\leq \frac{C}{\varepsilon^{d+1}_n}\int_{\hat{\varepsilon}_n\leq|x_i-y|\leq \tilde{\varepsilon}_n}dy
  \leq C\frac{(\tilde{\varepsilon}_n^{d}-\hat{\varepsilon}_n^{d})}{\varepsilon^{d+1}_n}
  \leq C\frac{\|Id-T_n\|_{L^\infty(\Omega)}}{\varepsilon^{2}_n}.
\end{align*}
Combining all estimates we arrive at
\begin{equation*}
  |I|\leq C\frac{\|Id-T_n\|_{L^\infty(\Omega)}}{\varepsilon^{2}_n}.
\end{equation*}
The proof is completed by using Lemma \ref{le:consistency_interior} for \eqref{eq:3.2}.
\end{proof}

\begin{lemma}\label{le:consistency_hypergraph_boundary}
Let $x\in\partial\Omega$ and $x_n\in \Omega_n$ such that $x_n\rightarrow x$ as $n\rightarrow \infty$.
If $\varphi\in C^3(\overline{\Omega})$, $\nabla \varphi(x)\neq 0$, and $n_0$ is sufficiently large,
then
\begin{equation*}
-L^{H_n}_{p,\varepsilon_n}\varphi(x_n)\leq 0 \mbox{ for any } n\geq n_0
\Longrightarrow
\min\left\{-L_p\varphi(x),\frac{\partial\varphi}{\partial\vec{n}}(x)\right\}\leq 0,
\end{equation*}
and
\begin{equation*}
-L_{p,\varepsilon}\varphi(x_\varepsilon)\geq 0  \mbox{ for any } n\geq n_0
\Longrightarrow
\max\left\{-L_p\varphi(x),\frac{\partial\varphi}{\partial\vec{n}}(x)\right\}\geq 0.
\end{equation*}
\end{lemma}
\begin{proof}
  Following the proof of Lemma \ref{le:consistency_hypergraph_interior}, we have
  \begin{equation*}
  L^{H_n}_{p,\varepsilon_n}\varphi(x_n)\leq L_{p,\varepsilon_n}\varphi(x_n)
  +C\frac{\|Id-T_n\|_{L^\infty(\Omega)}}{\varepsilon^{2}_n}.
\end{equation*}
The Lemma is proven by applying Lemma \ref{le:consistency_boundary} for $L_{p,\varepsilon_n}\varphi(x_n)$ and utilizing \eqref{eq:Tn}.
\end{proof}

\section{Discrete to continuum convergence}
In this section, we prove Theorem \ref{th:convergence}.
A key ingredient is a H\"older estimate for the solution $u_n$ of the hypergraph $p$-Laplacian equation \eqref{eq:2.1}.
The estimate allows us to apply the Arzel\`{a}--Ascoli compactness theorem, yielding the existence of a subsequence of $u_n$ that converges uniformly to the solution of equation \eqref{eq:2.2}. 
The proof is motivated by \cite{calder2018game}, where the game-theoretic $p$-Laplacian is studied. 
It follows the same strategy, with minor modifications.

\subsection{H\"older regularity of the hypergraph $p$-Laplacian equation}
The basic idea to establish the H\"older estimate is to utilize the comparison principle Proposition \ref{pr:comparison_principle_hypergraph} for the solution of the hypergraph $p$-Laplacian equation and a special solution of the form $|x-x_0|^\alpha +C(\varepsilon_n)$.
We need to show that $|x-x_0|^\alpha +C(\varepsilon_n)$ is indeed a solution for any $x_0\in\Omega_n$ if $\alpha$ and $C(\varepsilon_n)$ are appropriately selected.
The cases of $x$ away from $x_0$ and $x$ near $x_0$ are different and will be discussed separately in the following.

\begin{lemma}\label{le:holder_1}
  Let $\varepsilon_n\gg\sqrt{\delta_n}$, $p>d$, $0<\alpha<\frac{p-d}{p-1}$, $x_0\in \Omega_n$ with $\mbox{dist}(x_0,\partial\Omega)>(k+1)\varepsilon_n$, and $v(x)=|x-x_0|^\alpha$. If $n$ is sufficiently large, there exists a small constant $r>0$ and a large constant $m>0$, such that
  \begin{equation*}
    L^{H_n}_{p,\varepsilon_n}v(x_i)\leq 0,
  \end{equation*}
  for any $x_i\in\Omega_n\backslash\mathcal{O}$ with $ m\varepsilon_n<|x_i-x_0|\leq r$.
\end{lemma}
\begin{proof}
  For any $x\in\Omega$, if $x\neq x_0$,
  \begin{align*}
    \Delta^\rho v(x)
    &=\rho(x)\Delta v(x)+2\nabla\rho(x)\cdot\nabla v(x)
    \leq\rho(x)\left(\Delta v(x)+C|\nabla v(x)|\right) \\
    &=\rho(x)\left(\alpha(\alpha+d-2)|x-x_0|^{\alpha-2}
    +C\alpha|x-x_0|^{\alpha-1}\right)\\
    &\leq(\alpha+d-2+C|x-x_0|)\rho(x)\alpha|x-x_0|^{\alpha-2},
  \end{align*}
  and
  \begin{align*}
    \Delta_\infty^\rho v(x)=
      \rho(x)\alpha(\alpha-1)|x-x_0|^{\alpha-2},
  \end{align*}
  from which we have
  \begin{align}\label{eq:3.3}
  \begin{split}
    L_p v(x)
    &= \sigma_{\eta,1}\left(\Delta^\rho v(x) +(p-2)\Delta_\infty^\rho v(x)\right)\\
    &\leq C\left(\alpha+d-2+C|x-x_0|+(p-2) (\alpha-1)\right)|x-x_0|^{\alpha-2}.
  \end{split}
  \end{align}
  Let $x_i\in \Omega_{n}\backslash\mathcal{O}$ and $x_i\neq x_0$. Notice that
  \begin{align*}
    L^{H_n}_{p,\varepsilon_n}v(x_i)=&\frac{1}{2n\varepsilon_n^{2}}\sum\limits_{j=N+1}^n \eta^\chi_{\varepsilon_n}(|x_i-x_j|)\left(M^{k\varepsilon_n} v(x_j)+M_{k\varepsilon_n} v(x_j)-2v(x_i)\right)\\
    &+\frac{1}{2n\varepsilon_n^{2}}\sum\limits_{j=1}^N \eta^\chi_{\varepsilon_n}(|x_i-x_j|)\left(M^{k\varepsilon_n} v(x_j)+M_{k\varepsilon_n} v(x_j)-2v(x_i)\right).
  \end{align*}
  If $\mbox{dist}(x_i,\partial\Omega)>(k+1)\varepsilon_n$, applying Lemma \ref{le:consistency_hypergraph_interior} for the first summation and utilizing \eqref{eq:3.3}, we obtain that
  \begin{align*}
    L^{H_n}_{p,\varepsilon_n} v(x_i)
    &\leq L_p v(x_i)
    +C\left(|x_i-x_0|^{\alpha-3}+1\right)\varepsilon_n
    +\frac{C\|Id-T_n\|_{L^\infty(\Omega)}}{\varepsilon^{2}_n}+\frac{C}{n\varepsilon_n^{2}}\\
    &\leq C\left(\alpha+d-2+(p-2) (\alpha-1)+C|x_i-x_0|+\frac{\varepsilon_n}{|x_i-x_0|}\right. \\
    &\qquad\qquad\qquad\qquad \left.+ \left(\varepsilon_n
    +\frac{\|Id-T_n\|_{L^\infty(\Omega)}}{\varepsilon^{2}_n}+\frac{1}{n\varepsilon_n^{2}}\right)|x_i-x_0|^{2-\alpha}  \right)|x_i-x_0|^{\alpha-2}.
  \end{align*}
  The result still holds when $\mbox{dist}(x_i,\partial\Omega)\leq(k+1)\varepsilon_n$. In fact, we notice that $\nabla v(x_i)$ points outward from $\Omega$. Now we only need to use \eqref{eq:Lemma7} in Lemma \ref{le:consistency_boundary} and translate $L_{p,\varepsilon}$ into $L^{H_n}_{p,\varepsilon_n}$ by $T_n$.

  By the assumptions $p>d$ and $\varepsilon_n\gg\sqrt{\delta_n}$, we have
  $\alpha+d-2+(p-2)(\alpha-1)< 0$,  $\frac{\|Id-T_n\|_{L^\infty(\Omega)}}{\varepsilon^{2}_n}\rightarrow 0$, and $n\varepsilon_n^2\rightarrow\infty$.
  Recall that $m\varepsilon_n\leq |x_i-x_0|\leq r$.
  This yields
  \begin{equation*}
    L^{H_n}_{p,\varepsilon_n} v(x_i)
    \leq C\left(\alpha+d-2+(p-2)(\alpha-1)+Cr+\frac{1}{m}+C|m\varepsilon_n|^{2-\alpha}\right)
    |x_i-x_0|^{\alpha-2}\leq 0,
  \end{equation*}
  for a sufficiently small $r>0$ and a sufficiently large $m$.
\end{proof}

\begin{lemma}\label{le:holder_2}
  Let $p\geq 2$, $m>0$, $0<\alpha<1$, and $x_0\in\Omega_n$.
  Define
  \begin{equation*}
    v(x)=|x-x_0|^\alpha+M\varepsilon_n^\alpha\sum_{s=1}^\infty\beta^s
    \chi_{\{2|x-x_0|>(s-1)\varepsilon_n\}}.
  \end{equation*}
  If $n$ is sufficiently large, there exist a small constant $\beta>0$ and a large constant $M>0$, such that
  \begin{equation*}
    L^{H_n}_{p,\varepsilon_n}v(x_i)\leq 0,
  \end{equation*}
  for any $x_i\in\Omega_n\backslash\mathcal{O}$ with $ 0<|x_i-x_0|\leq (m+2k+2)\varepsilon_n$.
\end{lemma}
\begin{proof}
  Assume that $(l-1)\varepsilon_n<2|x_i-x_0|\leq l\varepsilon_n$ for an integer $1\leq l\leq 2(m+2\lceil k\rceil+2)$. Here $\lceil k\rceil$ denotes the smallest integer greater than or equal to $k$.
  For any $x\in B(x_i,(k+1)\varepsilon_n)\cap\Omega_n$,
  it follows from the definition that
  \begin{align*}
    v(x)-v(x_i)
    &\leq |x-x_0|^\alpha-|x_i-x_0|^\alpha+
    (2\lceil k\rceil+2)M\beta^{l+1}\varepsilon_n^\alpha\\
    &\leq \left(|x_i-x_0|+(k+1)\varepsilon_n\right)^\alpha-|x_i-x_0|^\alpha+
    (2\lceil k\rceil+2)M\beta^{l+1}\varepsilon_n^\alpha\\
    &\leq (m+3k+3)^\alpha\varepsilon_n^\alpha+(2\lceil k\rceil+2)M\beta^{l+1}\varepsilon_n^\alpha.
  \end{align*}
  Consequently,
  \begin{equation}\label{eq:3.4a}
    M^{k\varepsilon_n} v(x_j)-v(x_i)\leq (m+3k+3)^\alpha\varepsilon_n^\alpha+(2\lceil k\rceil+2)M\beta^{l+1}\varepsilon_n^\alpha,
  \end{equation}
   for any $x_j\in B(x_i,\varepsilon_n)\cap\Omega_n$.

  To estimate $M_{k\varepsilon_n} v(x_j)-v(x_i)$, let us define
  \begin{align*}
    R_{x_i}=
    \begin{cases}
      \Omega_n\cap B(x_i,\varepsilon_n)\cap B(x_0,|x_i-x_0|-\frac{\varepsilon_n}{2}), & \mbox{if } l\geq 3, \\
      \Omega_n\cap B(x_i,\varepsilon_n)\cap B(x_0,\varepsilon_n), & \mbox{if } l=1,2.
    \end{cases}
  \end{align*}
  There exists a constant $0<C<1$ such that $\frac{|R_{x_i}|}{|B(x_i,\varepsilon_n)\cap\Omega_n|}\geq C$.
  For any $x_j\in R_{x_i}$,
  \begin{align}\label{eq:3.4b}
    M_{k\varepsilon_n} v(x_j)-v(x_i)\leq
    \begin{cases}
      v(x_j)-v(x_i)\leq -M\beta^{l}\varepsilon_n^\alpha, & \mbox{if } l\geq 3, \\
      v(x_0)-v(x_i)\leq -M\beta^{l}\varepsilon_n^\alpha, & \mbox{if } l=1,2.
    \end{cases}
  \end{align}
  If $x_j\in (B(x_i,\varepsilon_n)\cap\Omega_n)\backslash R_{x_i}$,
  \begin{equation}\label{eq:3.4c}
    M_{k\varepsilon_n} v(x_j)-v(x_i)\leq M^{k\varepsilon_n} v(x_j)-v(x_i).
  \end{equation}
  Thus we deduce from  \eqref{eq:3.4a}--\eqref{eq:3.4c} that
  \begin{align*}
    2n\varepsilon_n^{2}L^{H_n}_{p,\varepsilon_n}v(x_i)
    =&\left(\sum\limits_{j\in R_{x_i}}+\sum\limits_{j\notin R_{x_i}}\right)
    \eta^\chi_{\varepsilon_n}(|x_i-x_j|)\left(M^{k\varepsilon_n} v(x_j)+M_{k\varepsilon_n} v(x_j)-2v(x_i)\right)\\
    \leq &\sum\limits_{j\in R_{x_i}}\eta^\chi_{\varepsilon_n}(|x_i-x_j|)
    \left((m+3k+3)^\alpha\varepsilon_n^\alpha+(2\lceil k\rceil+2)M\beta^{l+1}\varepsilon_n^\alpha -M\beta^{l}\varepsilon_n^\alpha\right)\\
    &+2\sum\limits_{j\notin R_{x_i}}\eta^\chi_{\varepsilon_n}(|x_i-x_j|)
    \left((m+3k+3)^\alpha\varepsilon_n^\alpha+(2\lceil k\rceil+2)M\beta^{l+1}\varepsilon_n^\alpha\right) \\
    \leq & \frac{|R_{x_i}|}{\varepsilon_n^{d-\alpha}}\left((m+3k+3)^\alpha+(2\lceil k\rceil+2)M\beta^{l+1} -M\beta^{l}\right) \\
    &+ \frac{4|(B(x_i,\varepsilon_n)\cap\Omega_n)\backslash R_{x_i}|}{\varepsilon_n^{d-\alpha}}\left((m+3k+3)^\alpha+(2\lceil k\rceil+2)M\beta^{l+1}\right).
  \end{align*}
  By selecting $\beta=\frac{|R_{x_i}|}{2(2\lceil k\rceil+2)(|R_{x_i}|+4|(B(x_i,\varepsilon_n)\cap\Omega_n)\backslash R_{x_i}|)}$ and $M$ be large enough, we obtain that the right-hand side of the above inequality is non-positive and complete the proof.
\end{proof}

\begin{proposition}\label{pr:holder}
  Let $\varepsilon_n\gg\sqrt{\delta_n}$, $p>d$, and $0<\alpha<\frac{p-d}{p-1}$. If $u_n$ is a solution of equation \eqref{eq:2.1} and $n$ is sufficiently large, then
  \begin{equation*}
    |u_n(x_i)-u_n(x_j)|\leq C|x_i-x_j|^\alpha +C\varepsilon_n^\alpha,
  \end{equation*}
  for any $x_i,x_j\in\Omega_n$.
\end{proposition}

\begin{proof}
Let $x_j\in\Omega_n$ be fixed and $\mbox{dist}(x_j,\partial\Omega)>(k+1)\varepsilon_n$. We define
\begin{align*}
  &v_1(x_i)=|x_i-x_j|^\alpha
  +M\varepsilon_n^\alpha\sum_{s=1}^\infty\beta^s\chi_{\{2|x_i-x_j|>(s-1)\varepsilon_n\}},\\
  &v_2(x_i)=|x_i-x_j|^\alpha+M\varepsilon_n^\alpha\frac{\beta(1-\beta^{2m+5})}{1-\beta},
\end{align*}
for any $x_i\in \Omega_{n}\backslash\mathcal{O}$
and
\begin{equation}\label{eq:3.5}
  v_{x_j}(x_i)=\min\left\{v_1(x_i), v_2(x_i) \right\}
  =\begin{cases}
     v_1(x_i), & \mbox{if } |x_i-x_j|\leq (m+k+1)\varepsilon_n, \\
     v_2(x_i), & \mbox{if } |x_i-x_j|\geq (m+k+1)\varepsilon_n.
   \end{cases}
\end{equation}
It follows from Lemmas \ref{le:holder_1}--\ref{le:holder_2} that
\begin{equation}\label{eq:3.5a}
  L^{H_n}_{p,\varepsilon_n}v_{x_j}(x_i)\leq 0,
\end{equation}
for any $|x_i-x_j|\leq r$.
In fact,
\begin{itemize}
  \item if $|x_i-x_j|\leq m\varepsilon_n$, $L^{H_n}_{p,\varepsilon_n}v_{x_j}(x_i)=L^{H_n}_{p,\varepsilon_n}v_1(x_i)\leq 0$.
  \item If $|x_i-x_j|\geq (m+2k+2)\varepsilon_n$, $L^{H_n}_{p,\varepsilon_n}v_{x_j}(x_i)=L^{H_n}_{p,\varepsilon_n}v_2(x_i)\leq 0$.
  \item If $(m+k+1)\varepsilon_n\leq |x_i-x_j|\leq (m+2k+2)\varepsilon_n$, we have $v_{x_j}(x_i)=v_2(x_i)$ and $v_{x_j}\leq v_2$ in $B(x_i,(k+1)\varepsilon_n)$. Thus $L^{H_n}_{p,\varepsilon_n}v_{x_j}(x_i)\leq L^{H_n}_{p,\varepsilon_n}v_2(x_i)\leq 0$.
  \item If $m\varepsilon_n\leq |x_i-x_j|\leq (m+k+1)\varepsilon_n$, we have $v_{x_j}(x_i)=v_1(x_i)$ and $v_{x_j}\leq v_1$ in $B(x_i,(k+1)\varepsilon_n)$. Thus $L^{H_n}_{p,\varepsilon_n}v_{x_j}(x_i)\leq L^{H_n}_{p,\varepsilon_n}v_1(x_i)\leq 0$.
\end{itemize}
This proves \eqref{eq:3.5a}.

Let us temporarily assume that $x_j\in\mathcal{O}$. There exists a constant $C>0$ such that
\begin{equation*}
  Cv_{x_j}(x_i)\geq \max f-\min f,
\end{equation*}
for any $x_i\in\mathcal{O}\backslash\{x_j\}$.
This yields
\begin{equation*}
  u_n(x_j)-Cv_{x_j}(x_i)\leq \min f\leq u_n(x_i)\leq \max f\leq u_n(x_j)+Cv_{x_j}(x_i),
\end{equation*}
for any $x_i\in\mathcal{O}$.
Notice that $u_n$ is a solution of equation \eqref{eq:2.1} and $v_{x_j}$ is a supersolution of equation \eqref{eq:2.1}, it follows from the comparison principle Proposition \ref{pr:comparison_principle_hypergraph} that
\begin{equation}\label{eq:3.5b}
  u_n(x_j)-Cv_{x_j}(x_i)\leq u_n(x_i)\leq u_n(x_j)+Cv_{x_j}(x_i),
\end{equation}
for any $x_i\in \Omega_n$.

We may rewrite \eqref{eq:3.5b} as
\begin{equation}\label{eq:3.5c}
  u_n(x_i)-Cv_{x_i}(x_j)\leq u_n(x_j)\leq u_n(x_i)+Cv_{x_i}(x_j).
\end{equation}
The maximum principle once again implies that \eqref{eq:3.5c} holds for any $x_i,x_j\in\Omega_n$
with $\mbox{dist}(x_i,\partial\Omega)>(k+1)\varepsilon_n$.
Since
\begin{equation*}
  |x_i-x_j|^\alpha\leq v_{x_i}(x_j)\leq |x_i-x_j|^\alpha+ M\varepsilon_n^\alpha\frac{\beta }{1-\beta},
\end{equation*}
we substitute it into \eqref{eq:3.5c} to find
\begin{equation*}
  u_n(x_i)-C\left(|x_i-x_j|^\alpha+\varepsilon_n^\alpha\right)
  \leq u_n(x_j)
  \leq u_n(x_i)+C\left(|x_i-x_j|^\alpha+\varepsilon_n^\alpha\right),
\end{equation*}
for any $x_i,x_j\in\Omega_n$
with $\mbox{dist}(x_i,\partial\Omega)>(k+1)\varepsilon_n$.
By the triangle inequality, the result also holds when $\mbox{dist}(x_i,\partial\Omega)\leq (k+1)\varepsilon_n$.
\end{proof}

\subsection{Proof of Theorem \ref{th:convergence}}
By Theorem \ref{th:hypergraph_existence} and Proposition \ref{pr:holder},
equation \eqref{eq:2.1} admits a unique solution $u_n$ such that
  \begin{equation*}
    |u_n(x_i)-u_n(x_j)|\leq C|x_i-x_j|^\alpha +C\varepsilon_n^\alpha,
  \end{equation*}
  for any $x_i,x_j\in\Omega_n$.
  Let $\tilde{u}_n(x)=u_n(T_n(x))$.
  For any $x,y\in\Omega$, we utilize the above estimate to find
  \begin{equation*}
    |\tilde{u}_n(x)-\tilde{u}_n(y)|\leq C|T_n(x)-T_n(y)|^\alpha +C\varepsilon_n^\alpha
    \leq C|x-y|^\alpha +C\|T_n-Id\|^\alpha_{L^\infty(\Omega)} +C\varepsilon_n^\alpha.
  \end{equation*}
By the Arzel\`{a}--Ascoli theorem (see the appendix in \cite{calder2015pde}),
there exists a subsequence of $\{\tilde{u}_n\}$ (still denoted by itself) and a function $u\in C^\alpha(\Omega)$ such that
\begin{equation*}
  \tilde{u}_n\rightarrow u\quad \mbox{uniformly in } \Omega,
\end{equation*}
  as $n\rightarrow\infty$.
  Since $\tilde{u}_n(x_i)= u_n(x_i)$ for $x_i\in\Omega_n$,
  it follows that
  \begin{equation}\label{eq:3.3.1}
    \lim_{n\rightarrow\infty}\max_{x_i\in\Omega_n}|u_n(x_i)-u(x_i)|=0.
  \end{equation}
We are left to show that $u$ is a viscosity solution of equation \eqref{eq:2.2}.

Let us first prove that $u$ is a viscosity subsolution.
Assume that $\varphi\in C^\infty(\overline{\Omega})$, $x_0\in\overline{\Omega}\backslash\mathcal{O}$, $\nabla\varphi(x_0)\neq 0$, and $u-\varphi$ has a strict global maximum at $x_0$.
  By \eqref{eq:3.3.1}, there exists a sequence of points $x_n\in\Omega_n\backslash\mathcal{O}$, such that ${u}_n-\varphi$ attains its global maximum at $x_n\in\Omega_n$ and $x_n\rightarrow x_0$ as $n\rightarrow\infty$.
  Consequently,
  \begin{equation*}
    u_n(x_n)-u_n(x_i)\geq \varphi(x_n)-\varphi(x_i),\quad \mbox{for any } x_i\in\Omega_n,
  \end{equation*}
  from which we deduce that
  \begin{equation*}
    M^{k\varepsilon_n}u_n(x_j)\leq M^{k\varepsilon_n}\varphi(x_j)+u_n(x_n)-\varphi(x_n),\quad
    M_{k\varepsilon_n}u_n(x_j)\leq M_{k\varepsilon_n}\varphi(x_j)+u_n(x_n)-\varphi(x_n),
  \end{equation*}
  for any $x_j\in\Omega_n$.
  This leads to
  \begin{align}\label{eq:3.3.2}
  \begin{split}
    0&=L^{H_n}_{p,\varepsilon_n}u_n(x_n)
    =\frac{1}{2n\varepsilon_n^{2}}\sum\limits_{j=1}^n \eta^\chi_{\varepsilon_n}(|x_n-x_j|)\left(M^{k\varepsilon_n} u_n(x_j)+M_{k\varepsilon_n} u_n(x_j)-2u_n(x_n)\right)\\
    &\leq\frac{1}{2n\varepsilon_n^{2}}\sum\limits_{j=1}^n \eta^\chi_{\varepsilon_n}(|x_n-x_j|)\left(M^{k\varepsilon_n} \varphi(x_j)+M_{k\varepsilon_n} \varphi(x_j)-2\varphi(x_n)\right)
    =L^{H_n}_{p,\varepsilon_n}\varphi(x_n).
  \end{split}
  \end{align}
  In the following, we consider the cases $x_0\in\Omega\backslash\mathcal{O}$ and $x_0\in\partial
  \Omega$ separately.

  If $x_0\in\Omega\backslash\mathcal{O}$, we utilize Lemma \ref{le:consistency_hypergraph_interior} to find
  \begin{equation*}
    L_p\varphi(x_0)=\lim_{n\rightarrow\infty}L_p\varphi(x_n)
    =\lim_{n\rightarrow\infty}L^{H_n}_{p,\varepsilon_n}\varphi(x_n)
    \geq \lim_{n\rightarrow\infty}L^{H_n}_{p,\varepsilon_n}u_n(x_n)=0,
  \end{equation*}
  which is equivalent to $-\Delta_p^\rho\varphi(x_0)\leq 0$.
  If $x_0\in\partial\Omega$, it follows from \eqref{eq:3.3.2} and Lemma \ref{le:consistency_hypergraph_boundary} that
  \begin{equation*}
    \min\left\{-L_p\varphi(x_0),\frac{\partial\varphi}{\partial\vec{n}}(x_0)\right\}\leq 0.
  \end{equation*}
  Again, we can replace $-L_p$ in the above with $-\Delta_p^\rho$.
  This verifies that $u$ is a viscosity subsolution of equation \eqref{eq:2.2}.
  The proof that $u$ is a viscosity supersolution is similar.

\section{Numerical experiments}
In this section, numerical experiments on high-dimensional data interpolation, including image inpainting and semi-supervised learning, are presented to evaluate the performance of graph and hypergraph equations.
Recall that $\Omega_n=\{x_i\}_{i=1}^n$ and $\{(x_i,y_i): x_i\in\mathcal{O},y_i\in\mathbb{R}\}$ are the data set and the labeled set, where $\mathcal{O}=\{x_i\}_{i=1}^N$, $1\leq N<n$.
The experimental setup follows that of \cite{shi2017weighted}.
For computational efficiency, we construct both the graph and the hypergraph using the $k$-nearest-neighbor ($k$-NN) approach, which differs slightly from the construction used in the theoretical analysis.
The similarity between two vertices $x_i,x_j\in\Omega_n$ is defined by
\begin{equation}\label{eq:5.0}
\tilde{w}_{i,j}
=\exp\left(-\frac{\|x_i-x_j\|^2}{\sigma(x_i)^2}\right),
\end{equation}
where $\sigma(x_i)$ denotes the distance from $x_i$ to its 20th nearest neighbor. To obtain symmetric weights, we use the edge weight
\[
w_{i,j}=\frac12\left(\tilde{w}_{i,j}+\tilde{w}_{j,i}\right).
\]
For each vertex $x_i$, $e_i\subset \Omega_i$ consists of the $50$ vertices with the largest values of $w_{i,j}$. The hyperedge is then given by
$E=\{e_i\}_{i=1}^n$.
We compare the performance of the graph Laplacian (GL) equation
\begin{equation}\label{eq:5.1}
  \begin{cases}
    \sum\limits_{j=1}^nw_{i,j}\chi_{e_i}(x_j)(u_n(x_j)-u_n(x_i))=0, & \mbox{if } x_i\in\Omega_n\backslash\mathcal{O}, \\
    u_n(x_i)=y_i, & \mbox{if } x_i\in \mathcal{O},
  \end{cases}
\end{equation}
with the hypergraph Laplacian (HGL) equation
\begin{align}\label{eq:5.2}
  \begin{cases}
    \sum\limits_{j=1}^n w_{i,j}\chi_{e_j}(x_i)
    \left(\max\limits_{x_k\in e_j}u_n(x_k)+\min\limits_{x_k\in e_j}u_n(x_k)-2u_n(x_i)\right)=0,\quad &x_i\in \Omega_{n}\backslash\mathcal{O},\\
    u_n(x_i)=y_i,\quad &x_i\in \mathcal{O}.
  \end{cases}
\end{align}
Equation \eqref{eq:5.1} corresponds to a linear system and can be solved efficiently \cite{shi2017weighted}.
For equation \eqref{eq:5.1}, we employ the fixed-point iteration
\begin{align*}
  u_n^{m+1}(x_i)=
  \begin{cases}
    \frac{1}{2\sum\limits_{j=1}^n w_{i,j}\chi_{e_j}(x_i)}
    \sum\limits_{j=1}^n w_{i,j}\chi_{e_j}(x_i)
    \left(\max\limits_{x_k\in e_j}u^m_n(x_k)+\min\limits_{x_k\in e_j}u^m_n(x_k)\right), & \mbox{if } x_i\in \Omega_{n}\backslash\mathcal{O}, \\
    x_i, & \mbox{if } x_i\in\mathcal{O}.
  \end{cases}
\end{align*}
which is convergent \cite{shi2024hypergraph1}. The iteration is terminated once $\|u_n^{m+1}-u_n^m\|_{\infty}\leq 1\times10^{-4}$.

\begin{figure}
  \centering
  \begin{tabular}{@{~}c@{~}c@{~}c@{~}c@{}}
  \begin{overpic}[width=.2\textwidth]{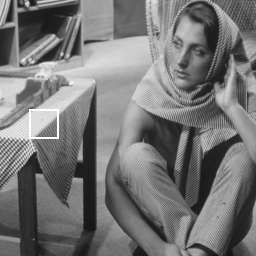}\put(50,0){}\end{overpic}&
  \begin{overpic}[width=.2\textwidth]{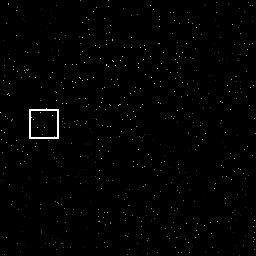}\put(50,0){\includegraphics[width=.1\textwidth]{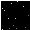}}\end{overpic}&
  \begin{overpic}[width=.2\textwidth]{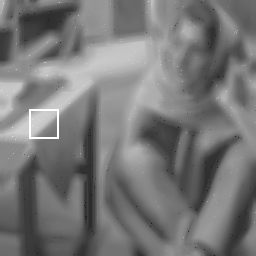}\put(50,0){\includegraphics[width=.1\textwidth]{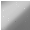}}\end{overpic}&
  \begin{overpic}[width=.2\textwidth]{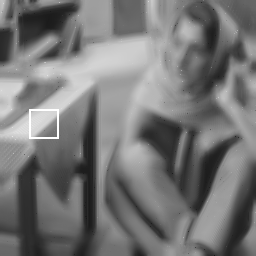}\put(50,0){\includegraphics[width=.1\textwidth]{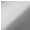}}\end{overpic}\\
  {\footnotesize\emph{Barbara}} & {\footnotesize\emph{1\% labeling rate}} & {\footnotesize\emph{GL, 20.80dB}} & {\footnotesize\emph{HGL, 23.30dB}} \\
  \begin{overpic}[width=.2\textwidth]{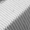}\put(50,0){}\end{overpic}&
  \begin{overpic}[width=.2\textwidth]{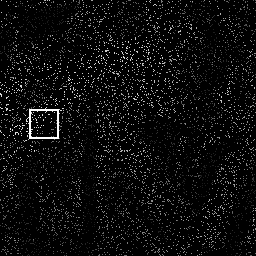}\put(50,0){\includegraphics[width=.1\textwidth]{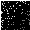}}\end{overpic}&
  \begin{overpic}[width=.2\textwidth]{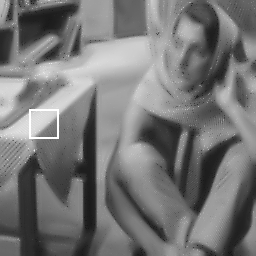}\put(50,0){\includegraphics[width=.1\textwidth]{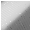}}\end{overpic}&
  \begin{overpic}[width=.2\textwidth]{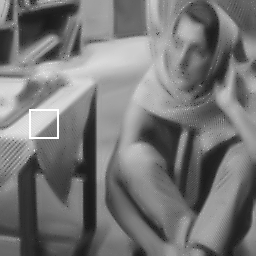}\put(50,0){\includegraphics[width=.1\textwidth]{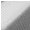}}\end{overpic}\\
    {\footnotesize\emph{zoomed-in region}} & {\footnotesize\emph{10\% labeling rate}} & {\footnotesize\emph{GL, 25.47dB}} & {\footnotesize\emph{HGL, 25.94dB}} \\
  \end{tabular}
    \caption{Inpainting results of GL and HGL for the test image \emph{Barbara}.
   From left to right: The test image, the observed pixels, inpainting results of \eqref{eq:5.1}, inpainting results of \eqref{eq:5.2}.
  }
    \label{fig:inpainting}
  \end{figure}

We first consider the image inpainting problem. Following \cite{shi2017weighted}, a point cloud is constructed from a discrete image
$
f\in\mathbb{R}^{n_1\times n_2}.
$
Let $p_{i,j}(f)$ denote the image patch centered at pixel $(i,j)$ with size $11\times 11$. 
A mirror extension of $f$ is used for pixels near the boundary.
The corresponding point cloud is defined by
\[
\mathcal{P}(f)=\{p_{i,j}(f):1\le i\le n_1,\;1\le j\le n_2\}
\subset\mathbb{R}^{11\times 11}.
\]
Let $n=n_1n_2$ and $\Omega_n=\mathcal{P}(f)$. Define
$u_n:\Omega_n\rightarrow\mathbb{R}$
by
$u_n(p_{i,j}(f))=f_{i,j}$,
where $f_{i,j}$ denotes the pixel intensity at $(i,j)$.
The image inpainting problem aims to recover the original image $f$ from the observed pixels
$\{f_{i,j}:(i,j)\in\mathcal O\}$,
where $\mathcal O\subset\{1,\ldots,n_1\}\times\{1,\ldots,n_2\}$.
This is an interpolation problem on the point cloud and solved by either equation \eqref{eq:5.1} or equation \eqref{eq:5.2}.

We assume that $\mathcal{P}(f)$ is known and the weights are computed from $\mathcal{P}(f)$ according to \eqref{eq:5.0}.
The labeled pixels $\mathcal{O}$ are selected uniformly at random. Figure \ref{fig:inpainting} presents the reconstruction results of GL and HGL under two different labeling rates for the test image \emph{Barbara}. 
When only $1\%$ of the pixels are labeled, HGL yields less spike artifacts and a substantially smoother reconstruction while preserving image details, resulting in a significantly higher PSNR. As the labeling rate increases to $10\%$, the performance gap between the two equations becomes much smaller, indicating that the advantage of HGL is most pronounced in the low-label regime.

We next consider semi-supervised learning on the well-known MNIST dataset, which consists of 70,000 grayscale images of handwritten digits ranging from $0$ to $9$. Each image is of size $28\times28$ and is regarded as a point in $\mathbb{R}^{784}$. The graph weights are computed according to \eqref{eq:5.0} in the same manner as in the image inpainting experiment.
Since the dataset contains ten classes, each label is encoded as a one-hot vector. The predicted class is then determined by the largest component of the interpolated vector.

Figure \ref{fig:SSL} reports the classification accuracy of GL and HGL under different labeling rates. The results are averaged over 10 independent runs. Similar to the image inpainting experiment, HGL consistently outperforms GL when the labeling rate is extremely low. As the labeling rate increases, the performance gap gradually diminishes.

\begin{figure}
  \centering
  \includegraphics[width=.5\textwidth]{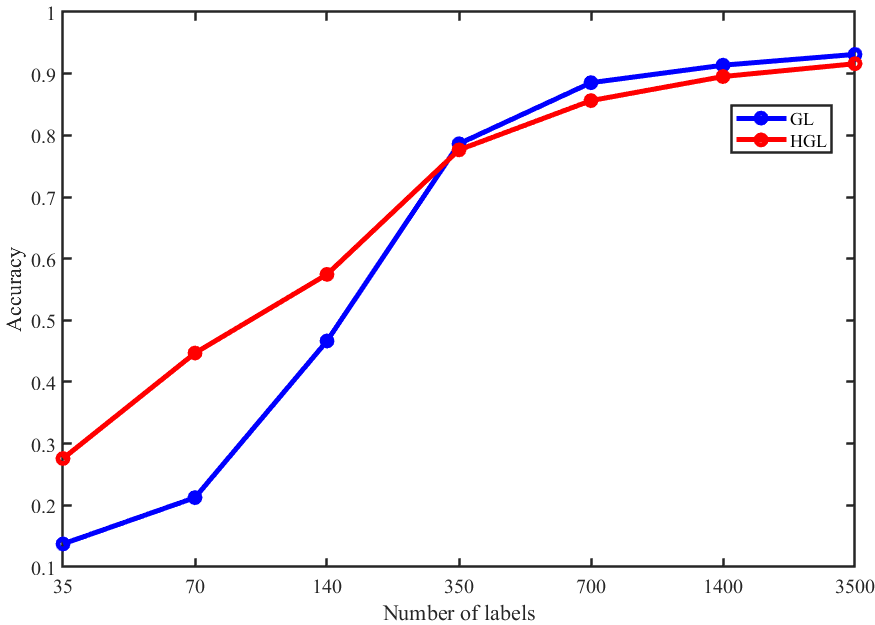}
    \caption{Semi-supervised learning accuracy of GL and HGL for the MNIST dataset.}
    \label{fig:SSL}
  \end{figure}

Finally, we compare the computational costs of the two equations. All experiments are performed in MATLAB R2020b on a desktop equipped with an Intel Core i7-8700 3.20 GHz CPU. For image inpainting, the average running times of GL and HGL are 5.3\,s and 9.3\,s, respectively. For semi-supervised learning, the corresponding running times are 38.0\,s and 47.4\,s, respectively. 
Although HGL requires additional computation due to the nonlinear fixed-point iteration, the increase in running time is moderate.

\section{Conclusion}
In this paper, we studied the continuum limit of a hypergraph $p$-Laplacian equation on point clouds for semi-supervised learning. 
It was proved that the hypergraph equation is a discretization of a weighted $p$-Laplacian equation with mixed Dirichlet and Neumann boundary conditions. 
Numerical experiments on high-dimensional data interpolation were also presented. 
Although the proposed hypergraph equation suppresses spike artifacts more effectively than the graph Laplacian equation in the extremely low-label regime, the reconstructed images still contain noticeable spikes. 
Developing more effective hypergraph-based interpolation models to further reduce these artifacts remains an interesting direction for future research.



\bibliographystyle{unsrt}
\bibliography{references}
\end{document}